\documentclass[12pt]{article}
\usepackage{amsmath}
\usepackage{amssymb}
\usepackage{mathrsfs}
\usepackage{url}

\usepackage{booktabs}
\usepackage{amsthm}
\usepackage{authblk}
\usepackage{algorithm}
\usepackage{algorithmic}
\usepackage{color}
\usepackage{graphicx}
\usepackage{caption}
\usepackage{subcaption}
\theoremstyle{definition}
\newtheorem{definition}{Definition}
\newtheorem{theorem}{Theorem}
\newtheorem{lemma}{Lemma}
\newtheorem{corollary}{Corollary}
\newtheorem{remark}{Remark}
\newtheorem{proposition}{Proposition}
\newcommand{\be}{\begin{eqnarray}}
\newcommand{\ee}{\end{eqnarray}}
\newcommand{\bea}{\begin{eqnarray*}}
\newcommand{\eea}{\end{eqnarray*}}
\usepackage{hyperref}
\usepackage[margin=1.2in]{geometry}

\DeclareMathOperator{\dist}{dist}
\DeclareMathOperator{\GL}{GL}

\newcommand{\tp}{\texttt{T}}

\newcommand{\pX}{\pmb{X}}
\newcommand{\pY}{\pmb{Y}}
\newcommand{\pU}{\pmb{U}}
\newcommand{\pV}{\pmb{V}}
\newcommand{\pZ}{\pmb{Z}}
\newcommand{\pM}{\pmb{M}}
\newcommand{\pQ}{\pmb{Q}}
\newcommand{\pR}{\pmb{R}}
\newcommand{\pH}{\pmb{H}}
\newcommand{\la}{\langle}
\newcommand{\ra}{\rangle}
\newcommand{\mbR}{\mathbb{R}}

\newcommand{\mP}{\mathcal{P}}
\newcommand{\mI}{\mathcal{I}}
\newcommand{\mL}{\mathcal{L}}

\newcommand{\oX}{\overline{\pmb{X}}}
\newcommand{\oY}{\overline{\pmb{Y}}}
\newcommand{\oZ}{\overline{\pmb{Z}}}
\newcommand{\oU}{\overline{\pmb{U}}}
\newcommand{\oV}{\overline{\pmb{V}}}
\newcommand{\oQ}{\overline{\pmb{Q}}}
\newcommand{\pSig}{\pmb{\Sigma}}
\newcommand{\pDe}{\pmb{\Delta}}
\newcommand{\pDeX}{\pmb{\Delta}_{\pmb{X}}}
\newcommand{\pDeY}{\pmb{\Delta}_{\pmb{Y}}}
\newcommand{\oSig}{\overline{\pmb{\Sigma}}}

\allowdisplaybreaks[4]

\begin{document}

\title{Nonconvex Deterministic Matrix Completion by Projected Gradient Descent Methods \footnote{This work was supported in part by the NSAF of China under grant number U21A20426,  National Natural Science Foundation of China under grant numbers 12071426, 11971427, 11901518, and the Project funded by China Postdoctoral Science Foundation 2023M733114.}}
\author{Hang Xu\footnote{School of Physics, Zhejiang University, Hangzhou 310027, P. R. China, E-mail address: hangxu@zju.edu.cn}}
\author{Song Li\footnote{School of Mathematical Science, Zhejiang University, Hangzhou 310027, P. R. China, E-mail address: songli@zju.edu.cn}}
\author{Junhong Lin\footnote{Corresponding author. Center for Data Science, Zhejiang University, Hangzhou 310027, P. R. China, E-mail address: junhong@zju.edu.cn}}
\affil{}		
\renewcommand*{\Affilfont}{\small\it}

\date{}

\maketitle
\begin{abstract}
We study deterministic matrix completion problem, i.e., recovering a low-rank matrix from a few observed entries where the sampling set is chosen as the edge set of a Ramanujan graph.
We first investigate projected gradient descent (PGD) applied to a Burer-Monteiro  least-squares problem and show that it  converges linearly  to the incoherent ground-truth with respect to the condition number $\kappa$ of ground-truth under a benign initialization and large samples.
We next apply the scaled variant of PGD to deal with the ill-conditioned case when $\kappa$ is large, and we show the algorithm converges at a linear rate independent of the condition number $\kappa$ under similar conditions. Finally, we provide numerical experiments  to corroborate our  results.

%
\end{abstract}


\section{Introduction}

 Low-rank matrix completion involves filling in missing entries of a low-rank matrix using  its few observed entries. 
 It has gained significant attention in various fields due to its wide range of applications, including
computer vision \cite{CS2004}, recommender systems \cite{KBV2009,zhang2021community}, quantum tomography \cite{GLF2010} and system identification \cite{LV2009}. The field of low-rank matrix completion has seen significant research efforts over the past decade, leading to the development of various algorithms and techniques \cite{CR2009,candes2010matrix,keshavan2010matrix,davenport2016overview,zhu2018global,chi2019nonconvex}.

\subsection{Matrix Completion}

The fundamental goal in matrix completion  is to recover a low-rank matrix  from its few observed entries. Specially,
let $\pmb{M}^* \in \mathbb{R}^{n_1\times n_2}$ be a rank-$r$ ($r<n_1 \wedge n_2$) matrix of interest.  Let $\Omega \subset [n_1] \times [n_2]$ be a (sampling) set of $m$ indices of the observed entries. Here, $[n]$ denotes the set $\{ 1, \dots, n\}$.  The corresponding observation is $\mathcal{P}_{\Omega}(\pmb{M}^*)$, where the sampling operator $\mathcal{P}_{\Omega}: \mathbb{R}^{n_1\times n_2} \to \mathbb{R}^{n_1\times n_2}$ is defined by
\begin{eqnarray}
[\mathcal{P}_{\Omega}(\pmb{X})]_{ij}=\begin{cases} \pmb{X}_{ij}, &(i, j)\in \Omega, \\ 0, &\text{otherwise}. \end{cases} \label{proj}
\end{eqnarray}
With these notations,
the aim is to recover the full matrix $\pmb{M}^*$ from its observation $\mathcal{P}_{\Omega}(\pmb{M}^*)$. 

A natural approach to address the above problem involves finding the matrix with the lowest rank that matches the given matrix $\pM^*$ on $\Omega$. This approach is known as the constrained rank minimization, which  
is NP-hard. 
An alternative method is to consider its convex relaxation,
nuclear norm minimization (NNM) \cite{oymak2011simplified,CR2009, CT2010, R2011}:
\bea
\min_{\pM\in \mbR^{n_1\times n_2}} \| \pmb{M} \|_*, 
\text{  s. t.  }  \mathcal{P}_{\Omega}(\pM) = \mathcal{P}_{\Omega}(\pM^*).
\eea
The nuclear norm is the sum of singular values of a matrix, and it is a convex function.
The optimization problem is thus  convex which can be efficiently solved using various convex optimization techniques, such as interior-point methods or gradient methods. 

Recently, a popular alternative approach for addressing the low-rank matrix recovery problem is the Burer-Monteiro factorization approach \cite{burer2003nonlinear}, which  is based on parametrizing $\pmb{M}^* = \pX^*{(\pY^*)}^{\tp}$ by two low-rank factors $\pX^* \in \mbR^{n_1\times r}$, $\pY^* \in \mbR^{n_2\times r}$. This method explicitly enforces the low-rank property by employing a factored representation of the matrix variable within the optimization formulation.

The advantage of the Burer-Monteiro factorization approach is that it can  
lead to more efficient optimization algorithms  for some large-scale problems
 compared to methods based on NNM, as it can reduce the number of variables involved in the optimization problem and avoid computational burdens of singular value decomposition (SVD) of large-scale matrices. 

Various optimization methods \cite{sun2016guaranteed, T2, CLL2020, TMC2021, NEURIPS2021_2f2cd5c7}  aiming at solving the following Burer-Monteiro factorization problem, 
\bea
\min_{\pmb{X} \in \mathbb{R}^{n_1\times r}, \pmb{Y} \in \mathbb{R}^{n_2\times r}}  \left\Vert \mathcal{P}_{\Omega}(\pmb{M}^*-\pmb{X}\pmb{Y}^\tp)\right\Vert_F^2,
\eea
either with or without a balancing term,  have been presented and their efficiency
has been demonstrated.




\subsection{Our Contributions}

In this paper, we study deterministic matrix completion, namely, matrix completion with the  sample set $\Omega$ chosen by a deterministic procedure \cite{HSS2014,ashraphijuo2017deterministic}.
Particularly, inspired by \cite{T1, BV2020}, the sample set $\Omega$ is chosen as the edge set of a Ramanujan graph.

The NNM has proven to be effective for the same setting in reconstructing the target matrix under certain mild incoherence conditions \cite{CR2009, T1} when the sample size  $m \gtrsim nr^2$ \cite{T1}. Here, $n=\max\{n_1,n_2\}$.
In this paper, we explore theoretical guarantees for PGD \cite{T2} and scaled PGD \cite{TMC2021} when applied to solve the (balanced) Burer-Monteiro factorization problems within the same framework of deterministic matrix completion as described in \cite{BV2020}.


Utilizing $\mathcal{O}(\kappa^2 n r^3)$ samples and assuming a mild incoherence condition, we demonstrate that  PGD exhibits linear convergence towards the ground-truth when initialized appropriately.
However, its convergence rate relies on the condition number $\kappa$ of the ground-truth, which could be slow for ill-conditioned matrices.
To address this issue, we use the scaled PGD variant, which we prove to converge at a linear rate independent of the condition number. This improvement is achieved with a sample size requirement of only $\mathcal{O}(\kappa^2nr^2(\kappa^2 \vee r))$.
The imposed incoherence condition and the sample complexity requirement are similar with those described in \cite{T1, BV2020} regarding the theoretical guarantees of NNM in deterministic matrix completion.

To the best of our knowledge, our results provide the first theoretical analysis of linear convergence for gradient-based methods utilizing matrix decomposition in the context of deterministic matrix completion.
In comparison to the classical approach of NNM as described in \cite{T1, BV2020}, our proposed methods offer a computational advantage by employing matrix factorization techniques, without compromising the theoretical guarantees. 


Our proofs incorporate proof insights for NNM in deterministic matrix completion \cite{T1, BV2020} and  (scaled) PGD in random matrix completion \cite{T2, TMC2021}.
Particularly, our proof for scaled PGD, relies on a novel alternative framework that is different from the approach employed in \cite{T2} within the context of random matrix completion. This alternative framework holds intrinsic value of its own and might be also useful for analyzing other matrix completion problems.

Finally, we conduct some numerical experiments to illustrate the performance and efficiency of the studied algorithms. 

\subsection{Related Work}



The earlier theoretical works on matrix completion considered the random sampling model and analyzed NNM \cite{recht2010guaranteed}. \cite{CR2009}  showed that the NNM model can recover the low-rank matrix if $m \gtrsim \mu n^{1.2}r \log n$ in Bernoulli sampling model. And a series of work \cite{CT2010, R2011, chen2015incoherence} focused on improving the sample complexity, where the best sample complexity is $\mathcal{O}(nr \log^2 n)$ \cite{chen2015incoherence}.
The solution of NNM can be computed by some famous algorithms, such as singular value thresholding method \cite{cai2010singular} and augmented lagrange multiplier method \cite{lin2010augmented}.
In order to reduce the storage space from $\mathcal{O}(n^2)$ to $\mathcal{O}(nr)$ for large-scale problem, a separate line of work \cite{keshavan2010matrix, sun2016guaranteed, T2, CLL2020, TMC2021} solved nonconvex formulations based on matrix decomposition. \cite{keshavan2010matrix} proposed a manifold gradient method for matrix completion. 
\cite{T2} showed that PGD linearly converges to the ground-truth with high probability  under $\mathcal{O}((\mu \vee \log n) \mu n r^2\kappa^2)$ random observations of a $\mu$-incoherent\footnote{The definition can be found in the following text.} matrix of rank-$r$ and condition number $\kappa$. 
\cite{TMC2021} applied scaled PGD to deal with moderately ill-conditioned matrices and showed that it converges at a rate independent of the condition number along with $\mathcal{O}((\mu \kappa^2 \vee \log n)\mu nr^2\kappa^2)$ random observations.

However, the assumption of uniform sampling is often invalid in practice \cite{meka2009matrix}. In some important applications such as data forecasting, the locations of missing entries cannot obey any non-degenerate distributions \cite{liu2019matrix}. Thus, exploring deterministic matrix completion methods can break through the limits of random sampling and is a valuable and significant problem.
There is only a little work \cite{HSS2014, T1, kiraly2015algebraic, pimentel2016characterization, ashraphijuo2017deterministic, BV2020} to choose a deterministic sampling set $\Omega$.
\cite{kiraly2015algebraic} chose the sample set $\Omega$ as the edge set of a bipartite graph and introduced an algebraic combinatorial theory to present probability-one algorithms to decide whether a particular entry of the matrix can be completed.
\cite{HSS2014, T1, BV2020} considered the deterministic matrix completion problem based on choosing the sampling set as the edge set of Ramanujan graphs.
\cite{HSS2014} applied max-norm minimization to give generalization error analysis. \cite{T1, BV2020} focused on NNM and derived sufficient conditions of recovery -- the sampling number should be at least $\mathcal{O}(\mu^2nr^2)$.

Ramanujan graphs are in some sense the best expander graphs with several applications in areas such as computer science \cite{hoory2006expander}, number theory and group theory \cite{lubotzky2012expander}.
\cite{lee2015sparsified} relied on the existence of  bipartite Ramanujan graphs of every degree to quickly approximate the complete graph in the fast solvers for Laplacian linear systems.
Ramanujan graphs play an important role in explaining the cutoff phenomenon of random walk undergoing a phase transition from being completely unmixed to completely mixed in the total variation norm \cite{lubetzky2011explicit}.
Ramanujan graphs have also been proposed as the basis for post-quantum elliptic-curve cryptography \cite{eisentrager2018supersingular}.

\subsection{Notations and Organization}
We introduce several notations which will be used throughout the paper. Unless specified, matrices and vectors are denoted by uppercase bold and lowercase bold letters, such as $\pmb{Z}$ and $\pmb{z}$. 
$\pmb{Z}_i$ is the $i$-row of $\pmb{Z}$, and
$\pmb{Z}_{ij}$ is the $(i, j)$-entry of $\pmb{Z}$ while $\pmb{z}_i$ is the $i$-th entry of $\pmb{z}$. The symbols $c$ and $C$ refer to constants which may not refer to the same number in each time. For any matrix $\pmb{Z}$, we use $\sigma_i(\pmb{Z})$ to denote its $i$-th largest singular value and $\sigma_i^* = \sigma_i(\pmb{M}^*)$. $\Vert \pmb{Z}\Vert$ denotes its operator norm, i.e., $\Vert \pmb{Z}\Vert = \sigma_1(\pmb{Z})$; $\Vert \pmb{Z}\Vert_*$ denotes its nuclear norm, i.e., $\Vert \pmb{Z}\Vert_* =\sum_{i=1}^{\text{rank}(\pmb{Z})}\sigma_i(\pmb{Z})$; $\Vert \pmb{Z}\Vert_F$ denotes its Frobenius norm, i.e., $\Vert \pmb{Z}\Vert_F = \sqrt{\sum_{ij}\vert \pmb{Z}_{ij}\vert^2}$. 
$\| \pmb{Z}\|_{2,\infty}$ denotes the largest $\ell_2$-norm of its rows as $\| \pmb{Z}\|_{2,\infty} = \max_{i}\|\pmb{Z}_i\|_2$.
For any vector $\pmb{z}$, $\Vert \pmb{z}\Vert_2$ denotes its Euclidean norm. For a block matrix $\pmb{Z}=\left[                 
  \begin{array}{ccc}   
   \pmb{U}_{a\times c}\\  
    \pmb{V}_{b\times c}\\  
  \end{array}
\right] \in \mathbb{R}^{(a+b)\times c}$, we set $\pmb{Z}_{\pmb{U}}=\pmb{U} \in \mathbb{R}^{a\times c}$ and $\pmb{Z}_{\pmb{V}}=\pmb{V} \in \mathbb{R}^{b\times c}$. The matrix inner product is defined as $\langle \pmb{U}, \pmb{V}\rangle = \text{tr} (\pmb{U}^\tp\pmb{V})$ and the vector inner product is $\langle \pmb{u}, \pmb{v}\rangle=\pmb{u}^\tp\pmb{v}$. $\pmb{I}_n$ denotes the $n\times n$ identity matrix and $\pmb{E}_n$ denotes the $n\times n$ all $1$' matrix. The subscript may be omitted when the order is clear from context. Besides, $\pmb{Z}^\tp$ stands for the transpose of $\pmb{Z}$.  For two non-negative real sequences $\{a_t\}_t$ and $\{b_t\}_t$, 
we write $b_t = \mathcal{O}(a_t)$ (or $b_t \lesssim a_t$) if $b_t \leq C a_t$.
$a \vee b$ means $\max\{a, b\}$ while $a\wedge b$ means $\min\{a, b\}$.

The paper is organized as follows. In the next section, we discuss Ramanujan graphs and  some assumptions of the deterministic matrix completion model. In Section \ref{sec2}, we present the two algorithms -- PGD and scaled PGD, and give the convergent results. We prove the convergence of scaled PGD in Section \ref{sec3}. In Section \ref{sec4}, we corroborate our theory through numerical experiments. We provide the convergent proof of PGD in the Appendix.

\section{Preliminaries}\label{sec:review}

In this section, we will present a very brief introduction of Ramanujan graphs and incoherence conditions in deterministic matrix completion.
\subsection{Ramanujan Graph}
Suppose that a bipartite graph $\mathcal{G}$ is associated with the sampling operator $\mathcal{P}_{\Omega}$ as follow. $\mathcal{G} = (V, E)$ where $V = \{ 1, 2, \dots, n_1\} \cup  \{ 1, 2, \dots, n_2\}$ is the vertex set and $E$ is the edge set where $(i, j) \in E$ iff $(i, j) \in \Omega$. $\pmb{G}$ is the adjacency matrix of $\mathcal{G}$ and belongs to $\{0,1\}^{n_1\times n_2}$. Thus, $\mathcal{P}_{\Omega}(\pmb{X}) = \pmb{G}\circ\pmb{X}$ where $\circ$ denotes the Hadamard product.
For our result, we need $\pmb{G}$ to have large spectral gap and the graph $\mathcal{G}$ to satisfy the following assumptions: 
\begin{itemize}
\item ${\rm \pmb{\mathscr{G}}_1} $:  Top singular vectors of $\pmb{G}$ are all $\sqrt{\frac{1}{n_1}}$'s vector or $\sqrt{\frac{1}{n_2}}$'s vector.\\
\item ${\rm \pmb{\mathscr{G}}_2} $:  $\sigma_1(\pmb{G}) = \sqrt{d_1d_2}$ and $\sigma_2(\pmb{G}) \leq \frac{C_0\sqrt{d_1}}{2} + \frac{C_0\sqrt{d_2}}{2}$.\\
\item ${\rm \pmb{\mathscr{G}}_3} $:  $(d_1, d_2)$-biregular, that is, per row contains precisely $d_1$ ones and per column contains precisely $d_2$ ones.
\end{itemize}
Note that $n_1\times d_1 = n_2 \times d_2$.
The above three assumptions are satisfied for a class of expander graphs called Ramanujan graphs.
\begin{definition}[Asymmetric Ramanujan graph, \cite{BV2020}]
Let $\sigma_1(\pmb{G})$, $\sigma_2(\pmb{G})$, $\dots$, $\sigma_n(\pmb{G})$ be the singular values of $\pmb{G}$ in decreasing order. Then, a $(d_1, d_2)$-biregular bipartite graph $\mathcal{G}$ is an asymmetric Ramanujan graph if $\sigma_1(\pmb{G}) = \sqrt{d_1d_2}$, $\sigma_2(\pmb{G}) \leq \sqrt{d_1-1} + \sqrt{d_2-1}$.
\end{definition}
Ramanujan graphs are well-studied in many literatures. Further details about Ramanujan graphs can be found in \cite{M2003}, \cite{DSV2003}. To some extend, Ramanujan graphs can replicate many of the desirable properties of random graphs, including expansion properties. Thus, we can utilize the above deterministic sampling model in place of the random one to obtain similar recovery performance. 
The improvement of deterministic method over probabilistic method for matrix completion can be found in \cite{burnwal2021deterministic} but only considering NNM. We will show that similar improvement  occurs on some numerical experiments in Section \ref{sec4}, considering alternative nonconvex methods, while achieving some computational advantages comparing with NNM. These motivate us to explore the nonconvex methods and their theoretical results for deterministic matrix
completion methods, based on choosing the sampling set as the edge set of a Ramanujan graph.
 

\subsection{Incoherence Conditions}
 It is observed in \cite{CR2009} that it is impossible to complete $\pmb{M}^*$ if $\pmb{M}^*$ is equal to zero in nearly all of its entries, unless all of its entries are observed. 
 In order to  avoid these pathological situations,  we first reproduce two standard incoherence assumptions on the SVD of $\pmb{M}^*= \pmb{U}^*\pmb{\Sigma}^* {(\pmb{V}^*)}^\tp$:
\begin{itemize}
\item ${\rm \pmb{\mathscr{A}}_1}$: $ \Vert \pmb{U}_i^*\Vert_2^2 \leq \frac{\mu r}{n_1}$, \quad $\Vert \pmb{V}_i^*\Vert_2^2 \leq \frac{\mu r}{n_2}$.\\
\item ${\rm \pmb{\mathscr{A}}_2}$: There is a constant $\delta_d$ such that   
\bea
&&\left\Vert \sum_{k\in S} \frac{n_1}{d_2}\pmb{U}_k^*{(\pmb{U}_k^*)}^\tp - \pmb{I}\right\Vert \leq \delta_d,\ \forall S \subset [n_1], \vert S \vert = d_2,\\
&&\left\Vert \sum_{k\in S} \frac{n_2}{d_1}\pmb{V}_k^*{(\pmb{V}_k^*)}^\tp - \pmb{I}\right\Vert \leq \delta_d,\ \forall S \subset [n_2], \vert S \vert = d_1.
\eea
\end{itemize}

${\rm \pmb{\mathscr{A}}_1}$ is a standard assumption required by most of the existing matrix completion results to ensure that the information of the row and column spaces of $\pmb{M}^*$ is not too concentrated on a small number of rows/columns. 
It is prevalent in random matrix completion problem \cite{CR2009, CT2010, R2011}.
${\rm \pmb{\mathscr{A}}_2}$, similar to the strong incoherence condition in \cite{CT2010}, is stricter than ${\rm \pmb{\mathscr{A}}_1}$. And \cite{T1} clarified that ${\rm \pmb{\mathscr{A}}_2}$ is necessary for deterministic case.

\section{Problem Setup and Main Results} \label{sec2}

\subsection{Nonconvex Model via Matrix Decomposition}
Let $\pmb{M}^* \in \mathbb{R}^{n_1 \times n_2}$ be a rank-$r$ matrix and $\pmb{M}^* = \pmb{U}^*\pmb{\Sigma}^* {(\pmb{V}^*)}^\tp$ be the SVD of $\pmb{M}^*$. Let $\sigma_1^* \geq \sigma_2^* \geq \dots \geq \sigma_r^* >0$ be the singular values of $\pmb{M}^*$ and the condition number of $\pmb{M}^*$ is $\kappa = \frac{\sigma_1^*}{\sigma_r^*}$. Suppose we have observed $\mathcal{P}_{\Omega}(\pmb{M}^*) = \pmb{M}^*\circ \pmb{G}$, where $\pmb{G}$ is the adjacency matrix of the bipartite graph $\mathcal{G}$, associated with $\Omega$. 
A direct and efficient procedure to recover the matrix $\pmb{M}^*$ is the least squares method:
\be
\min_{\pmb{X} \in \mathbb{R}^{n_1\times r}, \pmb{Y} \in \mathbb{R}^{n_2\times r}}  \frac{1}{2p} \left\Vert \mathcal{P}_{\Omega}(\pmb{M}^*-\pmb{X}\pmb{Y}^\tp)\right\Vert_F^2, \label{model_1}
\ee
where $p= \frac{d_1}{n_2} = \frac{d_2}{n_1} = \frac{\sqrt{d_1d_2}}{\sqrt{n_1n_2}}$. 
For a factorization of $\pmb{M} = \pmb{X} \pmb{Y}^\tp$, since $(\pX\pQ, \pY\pQ^{-\tp})$ is also a factorization of $\pmb{M}$ with any invertible matrix $\pQ \in \mbR^{r\times r}$,   \cite{pmlr-v48-tu16, T2, li2020nonconvex} added a regularizer to measure mismatch between $\pX$ and $\pY$:
\begin{eqnarray}
\min\limits_{\pmb{X} \in \mathbb{R}^{n_1\times r} \atop \pmb{Y} \in \mathbb{R}^{n_2\times r}}  \frac{1}{2p} \left\Vert \mathcal{P}_{\Omega}(\pmb{M}^*-\pmb{X}\pmb{Y}^\tp)\right\Vert_F^2 + \frac{\lambda}{4}\Vert \pmb{X}^\tp \pmb{X} - \pmb{Y}^\tp \pmb{Y}\Vert_F^2. \label{model_2}
\end{eqnarray}
Here, the second term is used for balancing or equal footing and $\lambda$ is an extra tuning parameter.

We will introduce two PGD methods based on nonconvex models and give the linear convergence analyses. 

\subsection{Projected Gradient Descent}
\indent In order to efficiently handle the asymmetric situation, we can lift $\pmb{M}^*$ to 
\begin{eqnarray*}
 \pmb{N}^* = \pmb{Z}^*{(\pmb{Z}^*)}^\tp =  \left[                 
  \begin{array}{ccc}   
   \pmb{U}^*\pmb{\Sigma}^*{(\pmb{U}^*)}^\tp & \pmb{M}^*\\  
    {(\pmb{M}^*)}^\tp & \pmb{V}^*\pmb{\Sigma}^*{(\pmb{V}^*)}^\tp\\  
  \end{array}
\right], 
\end{eqnarray*}
where $\pmb{Z}^* = \left[                 
  \begin{array}{ccc}   
   \pmb{U}^*\\  
    \pmb{V}^*\\  
  \end{array}
\right]{(\pmb{\Sigma}^*)}^{\frac{1}{2}}$. 
Since the entries in the top-left and bottom-right block are not caught in the sampling mechanism, we define that 
\begin{eqnarray*}
\underline{\pmb{G}} = \left[                 
  \begin{array}{ccc}   
   \pmb{0} & \pmb{G}\\  
   \pmb{G}^\tp &  -\pmb{0}\\  
  \end{array}
\right],
\end{eqnarray*}
associated with $\underline{\Omega}$, is the new projection matrix after lifting. 
The decomposition of $\pmb{N}^*$ is not unique, hence we introduce the regularizer $\frac{\lambda}{4}\Vert \pmb{Z}^\tp \pmb{D} \pmb{Z}\Vert_F^2$ to narrow down the range of solution similar to $\frac{\lambda}{4}\Vert \pmb{X}^\tp \pmb{X} - \pmb{Y}^\tp \pmb{Y}\Vert_F^2$ in \eqref{model_2}.
Then, we turn to consider the following loss function:
\begin{eqnarray}
\min_{\pmb{Z}} \mL_1(\pmb{Z})=\frac{1}{2p}\left\Vert \mathcal{P}_{\underline{\Omega}}(\pmb{Z}\pmb{Z}^\tp - \pmb{N}^*)\right\Vert_F^2 + \frac{\lambda}{4}\Vert \pmb{Z}^\tp \pmb{D} \pmb{Z}\Vert_F^2, \label{lo2}
\end{eqnarray}
where $\pmb{Z} = \left[                 
  \begin{array}{ccc}   
   \pmb{X}\\  
    \pmb{Y}\\  
  \end{array}
\right] = \left[                 
  \begin{array}{ccc}   
   \pmb{Z}_{\pmb{U}}\\  
    \pmb{Z}_{\pmb{V}}\\  
  \end{array}
\right] \in \mathbb{R}^{(n_1+n_2)\times r}$, $\pmb{D} = \left[                 
  \begin{array}{ccc}   
   \pmb{I}_{n_1} & \pmb{0}\\  
   \pmb{0} &  -\pmb{I}_{n_2}\\  
  \end{array}
\right]$.
Thus, the gradient of $ \mL_1(\pmb{Z})$ is given by 
\begin{eqnarray}
\nabla  \mL_1(\pmb{Z})
 &=& \frac{1}{p} \left[                 
  \begin{array}{ccc}   
   \pmb{0} & \pmb{K}\\  
   \pmb{K}^\tp &  \pmb{0}\\  
  \end{array}
\right]\pmb{Z}
 + \lambda \pmb{D}\pmb{Z}\pmb{Z}^\tp \pmb{D}\pmb{Z}. \label{lo2plus}
\end{eqnarray}
Here, $ \pmb{K}= \mathcal{P}_{\Omega}(\pmb{Z}_{\pmb{U}}(\pmb{Z}_{\pmb{V}})^\tp -\pmb{M}^*)$. 

For ensuring the iterates stay incoherent \cite{CW2015}, we define the set $\mathcal{C}_1$ of incoherent matrices
\begin{eqnarray*}
\mathcal{C}_1: = \left\{ \pmb{Z}: \Vert \pmb{Z}\Vert_{2,\infty} \leq \sqrt{\frac{2\mu r}{n_1\wedge n_2}}\Vert \pmb{Z}^0\Vert\right\},
\end{eqnarray*}
and the projection $\mathcal{P}_{\mathcal{C}_1}(\pmb{Z})$ to the set $\mathcal{C}_1$
\begin{eqnarray*}
\mathcal{P}_{\mathcal{C}_1}(\pmb{Z})_i =  \begin{cases}  \pmb{Z}_i, & \text{if}\ \Vert \pmb{Z}_i\Vert_2\leq \sqrt{\frac{2\mu r}{n_1\wedge n_2}}\Vert \pmb{Z}^0\Vert,\\ \frac{\pmb{Z}_i}{\Vert \pmb{Z}_i\Vert_2}\sqrt{\frac{2\mu r}{n_1 \wedge n_2}}\Vert \pmb{Z}^0\Vert, &\text{otherwise}. \end{cases}
\end{eqnarray*}
Then, PGD for solving (\ref{lo2}) can be stated as Algorithm \ref{alg:pgd}.

\begin{algorithm}
\caption{\small Projected Gradient Descent (PGD)} 
\label{alg:pgd}
\begin{algorithmic}
\STATE  \textbf{Input:}\quad $\eta$, $\lambda$, $\mathcal{P}_{\Omega}(\pmb{M}^*)$, MAXiter.

\STATE  \textbf{Initialize:}
\STATE  \quad Let $\pmb{U}^0\pmb{\Sigma}^0{(\pmb{V}^0)}^\tp$ be the top-$r$ SVD of 
\STATE  \begin{eqnarray*}
\pmb{M}^0 = \frac{1}{p}\mathcal{P}_{\Omega}(\pmb{M}^*).
\end{eqnarray*}
\STATE  \quad Set $\pmb{Z}^0 = \left[                 
  \begin{array}{ccc}   
   \pmb{U}^0\\  
    \pmb{V}^0\\  
  \end{array}
\right]{(\pmb{\Sigma}^0)}^{\frac{1}{2}}$ and $\pmb{Z}^1 = \mathcal{P}_{\mathcal{C}_1}(\pmb{Z}^0)$ for the initialization.

\textbf{Loop: } for $k = 1$ to MAXiter
\STATE  \begin{eqnarray*}
\pmb{Z}^{k+1} =  \mathcal{P}_{\mathcal{C}_1}\left(\pmb{Z}^k - \frac{\eta}{\Vert \pmb{Z}^*\Vert^2}\nabla  \mL_1(\pmb{Z}^k)\right).
\end{eqnarray*}
\STATE  \textbf{Output:} $\hat{\pmb{M}} = \pmb{Z}^{k+1}_{\pmb{U}}{(\pmb{Z}^{k+1}_{\pmb{V}})}^\tp$.
\end{algorithmic}
\end{algorithm}

\begin{remark}\label{rem3}
For the convenience of analysis, we choose  $\frac{\eta}{\Vert \pmb{Z}^*\Vert^2}$ as our stepsize where $\eta$ is a small constant and is tuned in practice. But in practice, $\Vert \pmb{Z}^*\Vert^2$ is unknown and could be replaced by $\Vert \pmb{Z}^0\Vert^2$. $\lambda$ will be chosen as $\lambda = \frac{1}{2}$ in the following analysis.
\end{remark}

Before presenting theoretical results, we give the definition of distance between $\pmb{Z}$ and the solution set of the problem
\begin{eqnarray*}
\mathcal{S} = \{ \widetilde{\pmb{Z}} \mid \widetilde{\pmb{Z}} =  \pmb{Z}^*\pmb{R}\ \text{for some orthogonal matrix } \pmb{R}\}.
\end{eqnarray*}
\begin{definition}
Let $\mathcal{S}$ be the solution set of the problem and $\pmb{Z}, \pmb{Z}^* \in \mathbb{R}^{(n_1+n_2)\times r}$, then the distance between $\pmb{Z}$ and $\pmb{Z}^*$  is
\begin{eqnarray*}
\dist(\pmb{Z}, \pmb{Z}^*) = \min_{\tilde{\pmb{Z}}\in \mathcal{S}} \Vert \pmb{Z} - \widetilde{\pmb{Z}} \Vert_F = \min_{\pmb{R}\pmb{R}^\tp =\pmb{R}^\tp \pmb{R} =\pmb{I}}\Vert \pmb{Z}-\pmb{Z}^*\pmb{R}\Vert_F.
\end{eqnarray*}
\end{definition}

Now, we present our convergence theorem that PGD can linearly converge to the ground-truth with the sample size $\mathcal{O}( nr^3)$. The proof is postponed to Appendix.
Nonconvex methods can deal better with the large-scale matrix completion problem than NNM studied in \cite{T1,BV2020}.
\begin{theorem} \label{t_1}
Assume the matrix $\pmb{M}^* \in \mathbb{R}^{n_1\times n_2}$ satisfies the assumptions ${\rm \pmb{\mathscr{A}}_1}$  and ${\rm \pmb{\mathscr{A}}_2}$  with $\delta_d \leq \frac{1}{64}$, the adjacency matrix $\pmb{G}$ of the graph satisfies the assumptions ${\rm \pmb{\mathscr{G}}_1}$, ${\rm \pmb{\mathscr{G}}_2}$  and ${\rm \pmb{\mathscr{G}}_3}$ . There exist some constants such that if $m \geq C \mu^2\kappa^2(n_1 \vee n_2)r^3$, when we choose $\lambda = \frac{1}{2}$ and $\eta = \mathcal{O}(\frac{1}{\mu^2r^2\kappa})$ in Algorithm \ref{alg:pgd}, the iteration $\{\pmb{Z}^k \}$ of Algorithm \ref{alg:pgd} satisfies
\begin{eqnarray*}
\dist(\pmb{Z}^k, \pmb{Z}^*)\leq \frac{1}{4}\left(1-\frac{\eta}{8\kappa}\right)^{\frac{k}{2}}\sqrt{\sigma_r^*}.
\end{eqnarray*}
Meanwhile, 
\bea
\left\| \pX^k {(\pY^k)}^\tp - \pM^* \right\|_F \leq \frac{\sqrt{2}}{4} \left(1-\frac{\eta}{8\kappa}\right)^{\frac{k}{2}} \left(\sqrt{\sigma_1^*\sigma_r^*} + \frac{1}{8}\sigma_r^*\right).
\eea
\end{theorem}

\begin{remark}\label{rem4}
It is worth noting that our sample size $\mathcal{O}( n r^3)$ for nonconvex deterministic matrix completion is $n$-optimal\footnote{ \cite{CT2010} stated that the minimum number of samples is on the order of $nr \log n$ rather than $nr$ because of a coupon collector’s effect if the observed locations are sampled at random. And we drop the term $\log n$ in deterministic matrix completion.},  compared with the sample size $\mathcal{O}(nr^2\log n)$ of nonconvex probabilistic matrix completion \cite{T2}.  
In practice, $r \ll \log n$, which means, $nr^3$ is less than $nr^2\log n$. For example, in quantum tomography, the dimensions are very huge but the rank is only one or two.
\end{remark}

\subsection{Scaled Projected Gradient Descent}
Theorem \ref{t_1} shows that PGD converges in $\mathcal{O}(\kappa \log(1/\epsilon))$ iterations with spectral initialization as long as sample size $m \gtrsim \mu^2 \kappa^2 n r^3$. For ill-conditioned matrices, the convergence rate of PGD slows down noticeably. Thus, motivated by \cite{TMC2021}, we adopt scaled PGD to optimize \eqref{model_1} and prove that it converges at a linear rate independent of condition number $\kappa$. 
To some extent, scaled PGD can accelerate the convergence for ill-conditioned  matrix problems.

We focus on the following nonconvex problem without any explicit regularization
\begin{eqnarray}
\min_{\pmb{Z} = [\pmb{X}^{\tp}\ \pmb{Y}^{\tp}]^{\tp} \in \mathbb{R}^{(n_1+n_2)\times r}} \mathcal{L}_2(\pmb{Z}) = \frac{1}{2p} \left\Vert \mathcal{P}_{\Omega}(\pmb{M}^*-\pmb{X}\pmb{Y}^\tp)\right\Vert_F^2. \label{model_noregular}
\end{eqnarray}
 In order to efficiently handle the asymmetric situation, we set $\pmb{X}^* = \pmb{U}^*(\pmb{\Sigma}^*)^{\frac{1}{2}}$, $\pmb{Y}^* = \pmb{V}^*(\pmb{\Sigma}^*)^{\frac{1}{2}}$ and $\pmb{Z}^* = \left[                 
  \begin{array}{ccc}   
   \pmb{X}^*\\  
    \pmb{Y}^*\\  
  \end{array}
\right]$.

For ensuring the iterates stay incoherent \cite{CW2015}, we define the set of incoherent matrices
\begin{eqnarray*}
 \left\{ \pmb{Z}= \left[                 
  \begin{array}{ccc}   
   \pmb{X}\\  
    \pmb{Y}\\  
  \end{array}
\right] \in \mathbb{R}^{(n_1+n_2)\times r}:  \begin{array}{ccc}   
   \Vert \pmb{X}(\pmb{\Sigma}^*)^{\frac{1}{2}}\Vert_{2,\infty} \leq \frac{B}{\sqrt{n_1}}\\  
    \Vert \pmb{Y}(\pmb{\Sigma}^*)^{\frac{1}{2}} \Vert_{2,\infty} \leq \frac{B}{\sqrt{n_2}}\\  
  \end{array} 
   \right\},
\end{eqnarray*}
as $\mathcal{C}_2$, 
where $B$ is a constant and we set $B= (1+\alpha)\sqrt{\mu r} \sigma_1^*$, $\alpha =0.1$.
And the projection $\mathcal{P}_{\mathcal{C}_2}(\cdot)$ is defined as 
\begin{align*}
\mathcal{P}_{\mathcal{C}_2}(\overline{\pmb{Z}})=&\mathop{\arg\min}_{\pmb{Z}} \left\{ \left\Vert \left(\pmb{X} - \overline{\pmb{X}}\right) \left[ (\oY)^\tp \oY \right]^{1\over2} \right\Vert_F^2 
 + \left\Vert 
\left(\pmb{Y} - \overline{\pmb{Y}}\right) \left[ (\oX)^\tp\oX \right]^{1\over2} \right\Vert_F^2  \right\} \\
& \text{s. t. } \sqrt{n_1}\left\Vert \pmb{X}\left[ (\oY)^\tp\oY\right]^{1\over2} \right\Vert_{2,\infty}  \leq B,\\
&\qquad \sqrt{n_2}\left\Vert \pmb{Y} \left[ (\oX)^\tp\oX\right]^{1\over2}\right\Vert_{2,\infty} \leq B.
\end{align*}
Here, $\overline{\pmb{Z}} = \left[                 
  \begin{array}{ccc}   
   \overline{\pmb{X}}\\  
    \overline{\pmb{Y}}\\  
  \end{array}
\right] \in \mathbb{R}^{(n_1+n_2)\times r}$, and we set $\pmb{Z} = \mathcal{P}_{\mathcal{C}_2}(\overline{\pmb{Z}})  = \left[                 
  \begin{array}{ccc}   
   \pmb{X}\\  
    \pmb{Y}\\  
  \end{array}
\right]$. The above projection admits a simple closed-form solution \cite{CW2015} as
\begin{eqnarray*}
\begin{cases}  \pmb{X}_{i} = \left( 1 \wedge \frac{B}{\sqrt{n_1}\left\Vert \overline{\pmb{X}}_{i}\left(\oY\right)^\tp \right\Vert_2} \right)\overline{\pmb{X}}_{i}, &1\leq i \leq n_1,\\ 
\pmb{Y}_{j} = \left( 1 \wedge \frac{B}{\sqrt{n_2}\left\Vert \overline{\pmb{Y}}_{j}\left(\oX\right)^\tp \right\Vert_2} \right)\overline{\pmb{Y}}_{j}, &1\leq j \leq n_2. 
\end{cases}
\end{eqnarray*}
Scaled PGD for (\ref{model_noregular}) can be stated as Algorithm \ref{alg:scaledPGD}.
We can find similar algorithms for probabilistic matrix completion in \cite{TMC2021}. The approach using the gradient descent directions scaled by $\pX^\tp\pX$ and $\pY^\tp\pY$ also appeared in \cite{tanner2016low}.

\begin{algorithm}[H]
\caption{\small Scaled Projected Gradient Descent (Scaled PGD)} 
\label{alg:scaledPGD}
\begin{algorithmic}
\STATE 
\STATE  \textbf{Input:}\quad $\eta$, $\mathcal{P}_{\Omega}(\pmb{M}^*)$, MAXiter.

\STATE \textbf{Initialize:}
\STATE \quad Let $\overline{\pmb{U}}{}^0\overline{\pmb{\Sigma}}{}^0{(\overline{\pmb{V}}{}^0)}^\tp$ be the top-$r$ SVD of 
\STATE  \begin{eqnarray*}
\pmb{M}^0 = \frac{1}{p}\mathcal{P}_{\Omega}(\pmb{M}^*).
\end{eqnarray*}
\STATE  \quad Set $\pmb{Z}^0 = \left[                 
  \begin{array}{ccc}   
   \pmb{X}^0\\  
    \pmb{Y}^0\\  
  \end{array}
\right] =\mathcal{P}_{\mathcal{C}_2}\left( \left[                 
  \begin{array}{ccc}   
   \oU{}^0{(\oSig{}^0)}^{\frac{1}{2}}\\  
    \oV{}^0{(\oSig{}^0)}^{\frac{1}{2}}\\  
  \end{array}
\right] \right).$ 

\STATE  \textbf{Loop: } for $k = 1$ to MAXiter
\STATE  \begin{eqnarray*}
\overline{\pmb{Z}}{}^{k+1} &=&  \left[                 
  \begin{array}{ccc}   
   \pmb{X}^k - \eta \nabla_{\pX}\mL_2(\pX^k, \pY^k) \left[{(\pY^k)}^\tp\pY^k\right]^{-1}\\  
   \pmb{Y}^k - \eta \nabla_{\pY}\mL_2(\pX^k, \pY^k) \left[{(\pX^k)}^\tp\pX^k\right]^{-1}\\  
  \end{array}
\right],\\
\pmb{Z}^{k+1} &=& \left[                 
  \begin{array}{ccc}   
   \pmb{X}^{k+1}\\  
    \pmb{Y}^{k+1}\\  
  \end{array}
\right] =   \mathcal{P}_{\mathcal{C}_2}\left(\overline{\pmb{Z}}{}^{k+1}\right).
\end{eqnarray*}
\STATE  \textbf{Output:} $\hat{\pmb{M}} = \pmb{X}^{k+1} {(\pmb{Y}^{k+1})}^\tp$.

\end{algorithmic}
\end{algorithm}

Here, if ${(\pY^k)}^\tp\pY^k$ and ${(\pX^k)}^\tp\pX^k$ are not invertible, we use generalized inverse.
Notice that $\nabla_{\pX}\mL_2(\pX, \pY) = \frac{1}{p}\mathcal{P}_{\Omega}\left(\pmb{X} \pmb{Y}^{\tp}-\pmb{M}^*\right)\pmb{Y}$, $\nabla_{\pY}\mL_2(\pX, \pY) = \frac{1}{p}\left[ \mathcal{P}_{\Omega}\left(\pmb{X}\pmb{Y}^{\tp}-\pmb{M}^*\right)\right]{}^{\tp}\pmb{X}$.
Inspired by the settings in \cite{TMC2021}, we give the definition of a new distance between $\pmb{Z}$ and $\pmb{Z}^*$.
\begin{definition}\label{def_dist*}
We define the distance between $\pmb{Z}$ and $\pmb{Z}^*$ as
\bea
\dist_*^2(\pmb{Z},\pmb{Z}^*) 
= \inf_{\pmb{Q}\in\GL(r)} \Big[ \left\Vert (\pmb{X}\pmb{Q} - \pmb{X}^*) (\pmb{\Sigma}^*)^{1\over 2} \right\Vert_F^2 
+ \left\Vert (\pmb{Y}\pmb{Q}^{-\tp} - \pmb{Y}^*)(\pSig^*)^{1\over 2} \right\Vert_F^2\Big],
\eea
where $\GL(r)$ denotes all the invertible matrices set of $\mathbb{R}^{r\times r}$. And we also define the optimal alignment matrix $\oQ$ between $\pmb{Z}$ and $\pmb{Z}^*$ as 
\be
\oQ = \mathop{\arg\min}_{\pmb{Q}\in \GL(r)} \Big[ \left\Vert (\pmb{X}\pmb{Q} - \pmb{X}^*)(\pmb{\Sigma}^*)^{1\over 2} \right\Vert_F^2  
 + \left\Vert (\pmb{Y}\pmb{Q}^{-\tp} - \pmb{Y}^*)(\pSig^*)^{1\over 2} \right\Vert_F^2\Big],
\label{ab_min_pro}
\ee
whenever the minimum is achieved.
\end{definition}
Similar to \cite{TMC2021}, we also state that when restricting to a small basin of $\pmb{Z}^*$, the optimal alignment matrix always exists and is well-defined. 

\begin{lemma}\label{lem:alignment1}
Fix any matrix $\pmb{Z} = \left[                 
  \begin{array}{ccc}   
   \pmb{X}\\  
    \pmb{Y}\\  
  \end{array}
\right]$. Suppose that 
$\dist_*(\pmb{Z}, \pmb{Z}^*)  < \sigma_r^*$,
then the minimizer of the above minimization problem \eqref{ab_min_pro} is attained at some $\oQ\in \GL(r)$, i.e. the optimal alignment matrix $\oQ$ between $\pmb{Z}$ and $\pmb{Z}^*$ exists.
\end{lemma}

Now, we present our main convergence theorem that scaled PGD can converge at a linear rate independent of $\kappa$.
\begin{theorem} \label{thm_scaledPGD}
Assume the matrix $\pmb{M}^* \in \mathbb{R}^{d_1\times d_2}$ satisfies the assumptions ${\rm \pmb{\mathscr{A}}_1}$ and ${\rm \pmb{\mathscr{A}}_2}$ with $\delta_d \leq \frac{1}{64}$, the adjacency matrix $\pmb{G}$ of the graph satisfies the assumptions ${\rm \pmb{\mathscr{G}}_1}$, ${\rm \pmb{\mathscr{G}}_2}$ and ${\rm \pmb{\mathscr{G}}_3}$. If $m \geq C\mu^2\kappa^2 (n_1 \vee n_2)r^2( \kappa^2 \vee r)$,
when we choose  $\eta \leq 0.145$ in Algorithm \ref{alg:scaledPGD}, the iteration $\{\pmb{Z}^k \}$ of Algorithm \ref{alg:scaledPGD} satisfies
\begin{eqnarray*}
\dist_{*}(\pmb{Z}^k, \pmb{Z}^*)\leq  \frac{1}{10} (1-1.6\eta + 11\eta^2)^{\frac{k}{2}}  \sigma_r^*.
\end{eqnarray*}
Moreover,
\bea
\Vert \pX^{k}(\pY^{k})^\tp - \pM^* \Vert_F \leq \frac{3}{20} (1-1.6\eta + 11\eta^2)^{\frac{k}{2}}  \sigma_r^*.
\eea
\end{theorem}

\begin{remark}
Here, we choose $\eta \leq 0.145$ to make $1-1.6\eta + 11\eta^2 < 1$.
And we use alternative local curvature and smoothness
conditions to give convergent analysis of scaled PGD for deterministic matrix completion in Section \ref{sec3}. Our proof may also simplify the existing analysis of scaled PGD for random matrix completion and could be also generalized to other low-rank matrix recovery problems.

\end{remark}

\section{Convergent Analysis of Scaled PGD} \label{sec3}
The proof follows mainly from the ideas in \cite{T2, TMC2021}. 
The main difference is the deficiency of randomness and 
we need to give some innovative lemmas only suitable for deterministic sampling, such as Lemma \ref{l_6} and Lemma \ref{l_7}. 
Since the proof of PGD for deterministic matrix completion is similar to the one for random matrix completion in \cite{T2}, we omit here and postpone it to Appendix. 
In this section, we only give the proof for scaled PGD. Besides, our approach is simpler than ones in \cite{TMC2021} and it also provides a new framework for the convergent analyses of some other scaled gradient descent in some canonical problems, such as matrix sensing, robust PCA.

\subsection{Proof Sketch}
In this part, we sketch the proof of convergent theorem of scaled PGD, highlighting the role of local curvature and smoothness condition.
Let $\pDeX = \oX\ \oQ - \pX^*$, $\pDeY = \oY(\oQ)^{-\tp} - \pY^*$ be the residual error, where $\oQ$ is the optimal alignment matrix between $\oZ$ and $\pZ^*$ with $\oZ= \left[                 
  \begin{array}{ccc}   
   \oX\\  
    \oY\\  
  \end{array}
\right] $, $\pmb{Z}^*= \left[                 
  \begin{array}{ccc}   
   \pmb{X}^*\\  
    \pmb{Y}^*\\  
  \end{array}
\right] \in \mathbb{R}^{(n_1+n_2)\times r}$. 
We first state that the loss function $\mathcal{L}_2(\pmb{Z})$ satisfies local curvature and smoothness condition \cite{CLS2015} which indicates the gradient of $\mathcal{L}_2(\pmb{Z})$ is well-behaved. Similar definition can also be found in \cite{T2}. 
\begin{proposition}[Local Curvature Condition] \label{prop_curvature}
For any $\oZ \in \mathcal{C}_2$ satisfying $\dist_*(\oZ,\pmb{Z}^*) \leq \frac{1}{10} \sigma_r^*$ and $\sqrt{n_1}\| \oX (\oY)^\tp \|_{2,\infty} \vee \sqrt{n_2}\| \oY (\oX)^\tp \|_{2,\infty} \leq (1+ \alpha)\sqrt{\mu r}\sigma_1^*$, if $\alpha = 0.1$, $\delta_d = \frac{1}{64}$ and $d_1\wedge d_2 \geq 100^2C_0^2\mu^2 r^2 \kappa^4$, we have 
\bea
\la \pDeX(\pSig^*)^{1\over2}, \nabla_{\pX}\mL_2(\oX, \oY)\left[ (\oY)^\tp\oY\right]^{-1}\oQ(\pSig^*)^{1\over2} \ra 
\geq 0.833 \| \pDeX (\pSig^*)^{1\over2} \|_F^2 - 0.023 \| \pDeY (\pSig^*)^{1\over2} \|_F^2
\eea
and 
\bea
\la \pDeY(\pSig^*)^{1\over2}, \nabla_{\pY}\mL_2(\oX, \oY)\left[(\oX)^\tp\oX\right]^{-1}(\oQ)^{-\tp}(\pSig^*)^{1\over2} \ra 
\geq 0.833 
\| \pDeY (\pSig^*)^{1\over2} \|_F^2 - 
0.023
\| \pDeX (\pSig^*)^{1\over2} \|_F^2.
\eea
\end{proposition}

\begin{proposition}[Local Smoothness Condition]\label{prop_smoothness}
For any $\oZ \in \mathcal{C}_2$ satisfying $\dist_*(\oZ,\pmb{Z}^*) \leq \frac{1}{10} \sigma_r^*$ and $\sqrt{n_1}\| \oX (\oY)^\tp \|_{2,\infty} \vee \sqrt{n_2}\| \oY (\oX)^\tp \|_{2,\infty} \leq (1+ \alpha)\sqrt{\mu r}\sigma_1^*$, if $\alpha = 0.1$, $\delta_d = \frac{1}{64}$ and $d_1\wedge d_2 \geq 100^2C_0^2\mu^2 r^2 \kappa^4$, we have
\bea
\|\nabla_{\pX}\mL_2(\oX, \oY)\left[(\oY)^\tp\oY \right]^{-1}\oQ (\pSig^*)^{1\over2}\|_F^2  
\leq 5.5\left[
\| \pDeX(\pSig^*)^{{1\over2}} \|_F^2 + 
\| \pDeY(\pSig^*)^{{1\over2}} \|_F^2 \right]
\eea
and
\bea
\|\nabla_{\pY}\mL_2(\oX, \oY)\left[ (\oX)^\tp\oX\right]^{-1}(\oQ)^{-\tp} (\pSig^*)^{1\over2}\|_F^2  
\leq 5.5\left[ 
\| \pDeX (\pSig^*)^{{1\over2}} \|_F^2 + 
\| \pDeY(\pSig^*)^{{1\over2}} \|_F^2\right]. 
\eea
\end{proposition}
We set $\nabla_{\oX}\mL_2 = \nabla_{\pX}\mL_2(\oX, \oY) \left[ (\oY)^\tp\oY\right]^{-1}$ and $\nabla_{ \oY}\mL_2 = \nabla_{\pY}\mL_2(\oX, \oY) \left[ (\oX)^\tp\oX\right]^{-1}$ for simplification. 
We say that $\mathcal{L}_2(\pmb{Z})$ satisfies the local curvature condition ${\bf LCC}(\alpha, \beta)$, if $$\la \pDeX (\pSig^*)^{1\over2}, \nabla_{\pX}\mL_2 \oQ (\pSig^*)^{1\over2} \ra \geq  \alpha \| \pDe_{\pX} (\pSig^*)^{1\over2} \|_F^2 - \beta \| \pDe_{\pY} (\pSig^*)^{1\over2} \|_F^2$$ 
and 
$$\la \pDe_{\pY}(\pSig^*)^{1\over2}, \nabla_{\pY}\mL_2(\oQ)^{-\tp} (\pSig^*)^{1\over2} \ra \geq \alpha \| \pDe_{\pY} (\pSig^*)^{1\over2} \|_F^2 - \beta \| \pDe_{\pX} (\pSig^*)^{1\over2} \|_F^2$$ 
hold for any $\oZ \in \mathcal{C}_2$ satisfying $\dist_*(\oZ,\pmb{Z}^*) \leq \frac{1}{10}\sigma_r^*$ and $\sqrt{n_1}\| \oX(\oY)^\tp \|_{2,\infty} \vee \sqrt{n_2}\| \oY (\oX)^\tp \|_{2,\infty} \leq 1.1\sqrt{\mu r}\sigma_1^*$. And $\mathcal{L}_2(\pmb{Z})$ satisfies the local smoothness condition ${\bf LSC}(\gamma)$, if $$\| \nabla_{\pX}\mL_2 \oQ (\pSig^*)^{1\over2}\|_F^2  
\leq \gamma(
\| \pDe_{\pX}(\pSig^*)^{{1\over2}} \|_F^2 + 
\| \pDe_{\pY}(\pSig^*)^{{1\over2}} \|_F^2)$$
 and 
$$\| \nabla_{ \pY}\mL_2 (\oQ)^{-\tp} (\pSig^*)^{1\over2}\|_F^2  
\leq \gamma(
\| \pDe_{\pX} (\pSig^*)^{{1\over2}} \|_F^2 + 
\| \pDe_{\pY} (\pSig^*)^{{1\over2}} \|_F^2)
$$
 hold for any $\oZ \in \mathcal{C}_2$ satisfying $\dist_*(\oZ,\pmb{Z}^*) \leq \frac{1}{10}\sigma_r^*$ and $\sqrt{n_1}\| \oX(\oY)^\tp \|_{2,\infty} \vee \sqrt{n_2}\| \oY(\oX)^\tp \|_{2,\infty} \leq 1.1\sqrt{\mu r}\sigma_1^*$.
Using these two conditions, one can show that the iterates contract.

\begin{proposition}\label{prop3}
Consider the iterate $\pmb{Z}^{k+1} =
\left(                
  \begin{array}{ccc}   
   \pmb{X}^{k+1} \\  
   \pmb{Y}^{k+1} \\  
  \end{array}
\right)
= \mathcal{P}_{\mathcal{C}_2}
\left(                
  \begin{array}{ccc}   
   \pmb{X}^k - \eta \nabla_{(\pX^k)}\mL_2\\  
   \pmb{Y}^k - \eta \nabla_{( \pY^k)}\mL_2\\  
  \end{array}
\right)$, and assume $\dist_*(\pZ^k,\pmb{Z}^*) \leq \frac{1}{10}\sigma_r^*$, $\sqrt{n_1}\| \pX^k(\pY^k){}^\tp \|_{2,\infty} \vee \sqrt{n_2}\| \pY^k(\pX^k){}^\tp \|_{2,\infty} \leq 1.1\sqrt{\mu r}\sigma_1^*$. If the loss function $\mathcal{L}_2(\pmb{Z})$ satisfies the local curvature condition ${\bf LCC}(\alpha, \beta)$ and local smoothness condition ${\bf LSC}(\gamma)$ with stepsize $\eta \leq \frac{\alpha - \beta}{\gamma}$, then
\bea
\dist_{*}^2(\pZ^{k+1}, \pZ^*) \leq C \dist_{*}^2(\pZ^{k}, \pZ^*),
\eea
with $C = 1-2(\alpha - \beta)\eta + 2\gamma \eta^2 < 1$. Moreover, 
\bea
\Vert \pX^{k+1}(\pY^{k+1})^\tp - \pM^* \Vert_F \leq 1.5\sqrt{C}\dist_{*}(\pZ^{k}, \pZ^*).
\eea
\end{proposition} 

\begin{proof}[Proof of Proposition \ref{prop3}]
We set $\pDeX^k = \pX^k\pQ^k - \pX^*$, $\pDeY^k = \pY^k(\pQ^k)^{-\tp} - \pY^*$, $\nabla_{(\pX^k)}\mL_2 = \nabla_{\pX}\mL_2(\pX^k, \pY^k) \left[(\pY^k)^\tp\pY^k\right]^{-1}$ and $\nabla_{(\pY^k)}\mL_2 = \nabla_{\pY}\mL_2(\pX^k, \pY^k) \left[(\pX^k)^\tp\pX^k \right]^{-1}$, where $\pQ^k$ is the optimal alignment matrix between $\pZ^k$ and $\pZ^*$.
Based on the iterates, we can get
\begin{align*}
&\dist_{*}^2(\pZ^{k+1}, \pZ^*) 
\leq \| (\pX^{k+1} \pQ^k -\pX^*)(\pSig^*)^{\frac{1}{2}}  \|_F^2 
+ \| [\pY^{k+1} (\pQ^k){}^{-\tp} -\pY^*](\pSig^*)^{\frac{1}{2}}  \|_F^2\\
\leq& \| [(\pX^k - \eta \nabla_{( \pX^k)}\mL_2)\pQ^k - \pX^*](\pSig^*)^{\frac{1}{2}}  \|_F^2 
+ \| [(\pY^k - \eta \nabla_{( \pY^k)}\mL_2) (\pQ^k)^{-\tp} - \pY^*] (\pSig^*)^{\frac{1}{2}} \|_F^2,
\end{align*}
since the the scaled projection $\mathcal{P}_{\mathcal{C}_2}$ is non-expansive and incoherent from Lemma \ref{non_expan}. 
Here, the first inequality holds based on Definition \ref{def_dist*}.

Since $\mathcal{L}_2(\pmb{Z})$ satisfies the local curvature condition ${\bf LCC}(\alpha, \beta)$ and local smoothness condition ${\bf LSC}(\gamma)$, we have
\begin{align*}
&\la \pDe_{\pX}^k(\pSig^*)^{1\over2}, \nabla_{(\pX^k)}\mL_2 \pQ^k (\pSig^*)^{1\over2} \ra 
+ \la \pDe_{\pY}^k(\pSig^*)^{1\over2}, \nabla_{( \pY^k)}\mL_2(\pQ^k){}^{-\tp}(\pSig^*)^{1\over2} \ra \\
\geq& (\alpha - \beta) \left[ \| \pDe_{\pX}^k (\pSig^*)^{1\over2} \|_F^2 + \| \pDe_{\pY}^k (\pSig^*)^{1\over2} \|_F^2 \right]
\end{align*}
and 
\bea
\| \nabla_{(\pX^k)}\mL_2 \pQ^k(\pSig^*)^{1\over2}\|_F^2 + \| \nabla_{( \pY^k)}\mL_2 (\pQ^k){}^{-\tp}(\pSig^*)^{1\over2}\|_F^2
 \leq 2\gamma\left[ \| \pDe_{\pX}^k(\pSig^*)^{{1\over2}} \|_F^2 + \|\pDe_{\pY}^k(\pSig^*)^{{1\over2}} \|_F^2 \right]
\eea
for any $\pZ^k \in \mathcal{C}_2$ satisfying $\dist_*(\pZ^k,\pmb{Z}^*) \leq \frac{1}{10}\sigma_r^*$, $\sqrt{n_1}\| \pX^k(\pY^k)^\tp \|_{2,\infty} \vee \sqrt{n_2}\| \pY^k(\pX^k){}^\tp \|_{2,\infty} \leq 1.1\sqrt{\mu r}\sigma_1^*$.

Then it is easy to find 
\begin{align*}
&\dist_{*}^2(\pZ^{k+1}, \pZ^*)\\
\leq& \dist_{*}^2(\pZ^{k}, \pZ^*) -2\eta \left[ \la \pDeX^k (\pSig^*)^{1\over2}, \nabla_{( \pX^k)}\mL_2\pQ^k (\pSig^*)^{1\over2} \ra 
+ \la \pDeY^k (\pSig^*)^{1\over2}, \nabla_{( \pY^k)}\mL_2(\pQ^k)^{-\tp} (\pSig^*)^{1\over2} \ra\right] \\
&+ \eta^2 \Big(\|\nabla_{( \pX^k)}\mL_2\pQ^k(\pSig^*)^{1\over2}\|_F^2 
+ \|\nabla_{( \pY^k)}\mL_2(\pQ^k)^{-\tp}(\pSig^*)^{1\over2}\|_F^2\Big) \\
\leq& [1-2(\alpha - \beta)\eta + 2\gamma \eta^2] \dist_{*}^2(\pZ^{k}, \pZ^*).
\end{align*}
Since $\pX^{k} (\pY^{k})^\tp - \pM^* = \pDeX^k(\pY^*)^\tp + \pX^* (\pDeY^k)^\tp + \pDeX^k(\pDeY^k)^\tp$, $\pX^* = \pU^*(\pSig^*)^{\frac{1}{2}}$ and $\pY^* = \pV^*(\pSig^*)^{\frac{1}{2}}$, the triangle inequality yields that 
\bea
\Vert \pX^{k}(\pY^{k})^\tp - \pM^* \Vert_F 
\leq \| \pDeX^k (\pSig^*)^{\frac{1}{2}} \|_F  +  \| \pDeY^k (\pSig^*)^{\frac{1}{2}} \|_F  + \| \pDeX^k(\pDeY^k)^\tp \|_F. 
\eea
After simple computation, we have 
\begin{align*}
 \| \pDeX^k (\pDeY^k)^\tp \|_F  
 =&  \frac{1}{2}\| \pDeX^k (\pSig^*)^{\frac{1}{2}}[\pDeY^k(\pSig^*)^{-\frac{1}{2}}]^\tp \|_F  + \frac{1}{2}\| \pDeY^k(\pSig^*)^{\frac{1}{2}}[\pDeX^k(\pSig^*)^{-\frac{1}{2}}]^\tp \|_F \\
 \leq& \frac{1}{2} \left(\|\pDeY^k(\pSig^*)^{-\frac{1}{2}} \| \vee \| \pDeX^k(\pSig^*)^{-\frac{1}{2}}\|\right) 
 \cdot \left(\| \pDeX^k(\pSig^*)^{\frac{1}{2}} \|_F + \| \pDeY^k(\pSig^*)^{\frac{1}{2}} \|_F\right).
\end{align*}
Since $\dist_*(\pZ^k, \pZ^*) \leq \frac{1}{10}\sigma_r^*$, we have
\begin{align*}
\|\pDeX^k(\pSig^*)^{-\frac{1}{2}} \|_F^2 + \| \pDeY^k(\pSig^*)^{-\frac{1}{2}}\|_F^2 
\leq \frac{1}{(\sigma_r^*)^2}\left( \|\pDeX^k(\pSig^*)^{\frac{1}{2}} \|_F^2 + \| \pDeY^k(\pSig^*)^{\frac{1}{2}}\|_F^2 \right) \leq \frac{1}{100}.
\end{align*}
It implies that 
\bea
\|\pDeX^k(\pSig^*)^{-\frac{1}{2}} \| \vee \| \pDeY^k(\pSig^*)^{-\frac{1}{2}}\| \leq \frac{1}{10}.
\eea
Thus, we can get 
\begin{align*}
\Vert \pX^{k}(\pY^{k})^\tp - \pM^* \Vert_F 
\leq& \frac{21}{20}\left[ \| \pDeX^k (\pSig^*)^{\frac{1}{2}} \|_F  +  \| \pDeY^k (\pSig^*)^{\frac{1}{2}} \|_F \right] \\
\leq& \frac{21\sqrt{2}}{20} \sqrt{\| \pDeX^k (\pSig^*)^{\frac{1}{2}} \|_F^2  +  \| \pDeY^k (\pSig^*)^{\frac{1}{2}} \|_F^2 } \\
\leq& 1.5\dist_*(\pZ^k, \pZ^*),
\end{align*}
which complete the proof.
\end{proof}

If the initialization behaves well, then Proposition \ref{prop3} indicates that the iterates will stay in the desired basin. Thus, scaled PGD can converge linearly to the ground-truth.

\begin{theorem}\label{thm:initialization2}
Suppose that $\pmb{M}^*$ is $\mu$-incoherent and the sample rate $p \geq 2000C_0^2\frac{\mu^2\kappa^2r^3}{n_1\wedge n_2}$, then the spectral initialization before projection $\overline{\pmb{Z}}{}^0 = \left[                 
  \begin{array}{ccc}   
   \oU{}^0(\oSig{}^0)^{\frac{1}{2}}\\  
    \oV{}^0(\oSig{}^0)^{\frac{1}{2}}\\  
  \end{array}
\right]$
satisfies
\bea
\dist_*(\overline{\pmb{Z}}{}^0, \pmb{Z}^*) \leq 0.1 \sigma_r^*.
\eea
\end{theorem}

\subsection{Useful Lemmas}
In this section, we state several useful lemmas that are used in the proof of main results. Because of the deficiency randomness, we utilize the property of Ramanujan graphs to get some results about deterministic sampling.
Recalling that $T= \{ \pmb{U}^*\pmb{X}^\tp + \pmb{Y}(\pmb{V}^*)^\tp | \pX \in \mathbb{R}^{n_1\times r},  \pY \in \mathbb{R}^{n_2\times r}\}$,  \cite{T1} proved the injectivity of operator $\mathcal{P}_{\Omega}$ on the subspace $T$.

\begin{lemma} \cite[Lemma 7.1]{T1} \label{l_4}
Let $\pmb{G}$ satisfy the assumptions ${\rm \pmb{\mathscr{G}}_1} $, ${\rm \pmb{\mathscr{G}}_2} $ and ${\rm \pmb{\mathscr{G}}_3} $, $\pmb{M}^* = \pmb{U}^*\pmb{\Sigma}^* (\pmb{V}^*)^\tp$ satisfy the assumptions ${\rm \pmb{\mathscr{A}}_1} $ and ${\rm \pmb{\mathscr{A}}_2} $. 
Then for any matrix $\pmb{Z} \in T$, we have
\begin{eqnarray*}
\left\Vert \frac{1}{p}\mathcal{P}_T\mathcal{P}_\Omega(\pmb{Z})-\pmb{Z}\right\Vert_F\leq \tilde{\delta}\Vert \pmb{Z}\Vert_F,
\end{eqnarray*}
where $\tilde{\delta} = \sqrt{2\Big(\delta_d^2+\frac{C_0^2\mu^2r^2}{d_1\wedge d_2}\Big)}$ and $C_0$ is the  constant in the assumption ${\rm \pmb{\mathscr{G}}_2} $. 
\end{lemma}

Then, we can get the restricted strong convexity and smoothness of the sampling operator $\mathcal{P}_{\Omega}$ for matrices in $T$ from Lemma \ref{l_4}.
\begin{lemma} \label{l_5}
Let $\pmb{G}$ satisfy the assumptions ${\rm \pmb{\mathscr{G}}_1} $, ${\rm \pmb{\mathscr{G}}_2} $ and ${\rm \pmb{\mathscr{G}}_3} $, $\pmb{M}^* = \pmb{U}^*\pmb{\Sigma}^* (\pmb{V}^*)^\tp$ satisfy the assumptions ${\rm \pmb{\mathscr{A}}_1} $ and ${\rm \pmb{\mathscr{A}}_2} $. Then for all $\pmb{Z} \in T$, we have 
\begin{eqnarray*}
(1-\tilde{\delta})\Vert \pmb{Z}\Vert_F^2 \leq \frac{1}{p}\Vert \mathcal{P}_\Omega(\pmb{Z})\Vert_F^2  \leq (1+\tilde{\delta})\Vert \pmb{Z}\Vert_F^2.
\end{eqnarray*}
Moreover, for all $\pmb{Z}_1, \pmb{Z}_2 \in T$,
\begin{eqnarray*}
\left\vert \frac{1}{p}\langle \mathcal{P}_\Omega(\pmb{Z}_1), \mathcal{P}_\Omega(\pmb{Z}_2)\rangle - \langle \pmb{Z}_1, \pmb{Z}_2\rangle\right\vert \leq \tilde{\delta} \Vert \pmb{Z}_1\Vert_F\Vert \pmb{Z}_2\Vert_F.
\end{eqnarray*}
Here, $\tilde{\delta} = \sqrt{2\Big(\delta_d^2+\frac{C_0^2\mu^2r^2}{d_1\wedge d_2}\Big)}$.
\end{lemma}

\begin{proof}
From Lemma \ref{l_4}, we can get
\begin{align*}
\left\Vert  \frac{1}{p}\mathcal{P}_T\mathcal{P}_\Omega(\pmb{Z})\right\Vert_F \leq (1+\tilde{\delta})\Vert \pmb{Z}\Vert_F.
\end{align*}
For all $\pmb{Z} \in T$, it comes that
\begin{align*}
\frac{1}{p}\Vert \mathcal{P}_\Omega(\pmb{Z})\Vert_F^2 
 =&\frac{1}{p} \langle \mathcal{P}_\Omega\mathcal{P}_T(\pmb{Z}), \mathcal{P}_\Omega\mathcal{P}_T(\pmb{Z}) \rangle 
 = \langle \frac{1}{p}\mathcal{P}_T\mathcal{P}_\Omega\mathcal{P}_T(\pmb{Z}), \pmb{Z} \rangle \\
\leq&  \left\Vert  \frac{1}{p}\mathcal{P}_T\mathcal{P}_\Omega\mathcal{P}_T(\pmb{Z})\right\Vert_F \left\Vert \pmb{Z}\right\Vert_F
\leq (1+\tilde{\delta})\Vert \pmb{Z}\Vert_F^2
\end{align*}
and 
\begin{align*}
\frac{1}{p}\Vert \mathcal{P}_\Omega(\pmb{Z})\Vert_F^2 =& \langle \frac{1}{p}\mathcal{P}_T\mathcal{P}_\Omega\mathcal{P}_T(\pmb{Z}), \pmb{Z} \rangle 
 = \langle \frac{1}{p}\mathcal{P}_T\mathcal{P}_\Omega\mathcal{P}_T(\pmb{Z}) - \mathcal{P}_T(\pmb{Z}) + \mathcal{P}_T(\pmb{Z}), \pmb{Z} \rangle\\
\geq&-\left\Vert \frac{1}{p}\mathcal{P}_T\mathcal{P}_\Omega(\pmb{Z})-\pmb{Z}\right\Vert_F\Vert \pmb{Z}\Vert_F+\Vert \pmb{Z}\Vert_F^2
\geq (1-\tilde{\delta})\Vert \pmb{Z}\Vert_F^2.
\end{align*} 
Thus, for all $\pmb{Z} \in T$, we have 
\begin{eqnarray}
(1-\tilde{\delta})\Vert \pmb{Z}\Vert_F^2 \leq \frac{1}{p}\Vert \mathcal{P}_\Omega(\pmb{Z})\Vert_F^2  \leq (1+\tilde{\delta})\Vert \pmb{Z}\Vert_F^2.  \label{5p1}
\end{eqnarray}
For all $\pmb{Z}_1, \pmb{Z}_2 \in T$,  we set $\pmb{Z}_1' = \frac{\pmb{Z}_1}{\Vert \pmb{Z}_1\Vert_F}$ and  $\pmb{Z}_2' = \frac{\pmb{Z}_2}{\Vert \pmb{Z}_2\Vert_F}$, then $\pmb{Z}_1' + \pmb{Z}_2', \pmb{Z}_1' - \pmb{Z}_2' \in T$. Based on (\ref{5p1}), we have
\begin{align*}
\frac{1}{p}\langle \mathcal{P}_\Omega(\pmb{Z}_1'), \mathcal{P}_\Omega(\pmb{Z}_2')\rangle
=& \frac{1}{4p}\left( \Vert \mathcal{P}_\Omega(\pmb{Z}_1' + \pmb{Z}_2')\Vert_F^2 - \Vert \mathcal{P}_\Omega(\pmb{Z}_1' - \pmb{Z}_2')\Vert_F^2\right)\\
\leq& \frac{1}{4} \left\{ (1+\tilde{\delta})\Vert \pmb{Z}_1'+\pmb{Z}_2'\Vert_F^2-(1-\tilde{\delta})\Vert \pmb{Z}_1'-\pmb{Z}_2'\Vert_F^2\right\}\\
=&  \frac{1}{4}\left\{ 2\tilde{\delta} (\Vert \pmb{Z}_1'\Vert_F^2 +\Vert \pmb{Z}_2'\Vert_F^2) + 4\langle \pmb{Z}_1', \pmb{Z}_2'\rangle\right\}\\
\overset{(i)}{=}& \tilde{\delta} + \langle \pmb{Z}_1', \pmb{Z}_2'\rangle,
\end{align*} 
where  equation $(i)$ holds because $\Vert \pmb{Z}_1'\Vert_F^2 = \Vert \pmb{Z}_2'\Vert_F^2 = 1$. 
Thus,
\begin{align*}
\frac{1}{p}\langle \mathcal{P}_\Omega(\pmb{Z}_1), \mathcal{P}_\Omega(\pmb{Z}_2)\rangle =& \frac{1}{p}\Vert \pmb{Z}_1\Vert_F  \Vert \pmb{Z}_2\Vert_F\langle \mathcal{P}_\Omega(\pmb{Z}_1'), \mathcal{P}_\Omega(\pmb{Z}_2')\rangle\\
\leq&  \tilde{\delta} \Vert \pmb{Z}_1\Vert_F  \Vert \pmb{Z}_2\Vert_F + \langle \pmb{Z}_1, \pmb{Z}_2\rangle.
\end{align*} 
In a similar way, we can get
\begin{align*}
\frac{1}{p}\langle \mathcal{P}_\Omega(\pmb{Z}_1), \mathcal{P}_\Omega(\pmb{Z}_2)\rangle \geq - \tilde{\delta}\Vert \pmb{Z}_1\Vert_F  \Vert \pmb{Z}_2\Vert_F + \langle \pmb{Z}_1, \pmb{Z}_2\rangle.
\end{align*} 
Then it goes that 
\begin{align*}
\left\vert \frac{1}{p}\langle \mathcal{P}_\Omega(\pmb{Z}_1), \mathcal{P}_\Omega(\pmb{Z}_2)\rangle- \langle \pmb{Z}_1, \pmb{Z}_2\rangle\right\vert \leq  \tilde{\delta}\Vert \pmb{Z}_1\Vert_F  \Vert \pmb{Z}_2\Vert_F.
\end{align*} 
\end{proof}
 
Lemma \ref{l_6} upper bounds the spectral norm of the adjacency matrix $\pmb{G}$ of a Ramanujan graph $\mathcal{G}$. Similar result of random Erd$\ddot{\text{o}}$s - R\'enyi graph can be found in Lemma 9 of \cite{CW2015}.
\begin{lemma} \label{l_6}
Suppose that $\pmb{A} \in \mathbb{R}^{n_1\times n_2}$  is the adjacency matrix of $(d_1, d_2)$-biregular graph that satisfies the assumptions ${\rm \pmb{\mathscr{G}}_1} $, ${\rm \pmb{\mathscr{G}}_2} $, ${\rm \pmb{\mathscr{G}}_3} $.
and $\Omega$ is the set of edges of this graph. 
Then for all $\pmb{x}\in \mathbb{R}^{n_1}$, $\pmb{y}\in \mathbb{R}^{n_2}$, it holds that
\begin{eqnarray*}
\frac{1}{p}\sum_{(i, j)\in \Omega}\pmb{x}_i\pmb{y}_j \leq \Vert \pmb{x}\Vert_1\Vert \pmb{y}\Vert_1 + \frac{C_0}{2} \frac{\sqrt{n_1}+\sqrt{n_2}}{\sqrt{p}} \Vert \pmb{x}\Vert_2\Vert \pmb{y}\Vert_2.
\end{eqnarray*} 
\end{lemma}
\begin{proof}
Let $\pmb{x}_i = \pmb{x}_0 + \pmb{x}_i'$ where $\sum_{i=1}^{n_1} \pmb{x}_i' =0$ and $\pmb{x}_0$ is a constant satisfying $\pmb{x}_0 = \frac{\sum_{i=1}^{n_1}\pmb{x}_i}{n_1}$. Thus, 
\begin{eqnarray*}
\sum_{(i, j)\in \Omega}\pmb{x}_i \pmb{y}_j = \sum_{(i, j)\in \Omega}\pmb{x}_i'\pmb{y}_j + \pmb{x}_0\sum_{j=1}^{n_2} \text{deg}(j)\pmb{y}_j.
\end{eqnarray*} 
Here, $\text{deg}(j)$ denotes the number of nonzero entries in the $j$-th column of $\pmb{A}$.\\
\indent The first term is upper bounded by
\begin{align*}
&\sum_{(i, j)\in \Omega}\pmb{x}_i'\pmb{y}_j \\
=& \frac{\pmb{x}'{}^\top}{\Vert \pmb{x}'\Vert_2}\pmb{A}\frac{\pmb{y}}{\Vert \pmb{y}\Vert_2} \cdot (\Vert \pmb{x}'\Vert_2\cdot\Vert \pmb{y}\Vert_2) \overset{(i)}{\leq} \sigma_2(\pmb{A})\Vert \pmb{x}'\Vert_2\Vert \pmb{y}\Vert_2 \\
\overset{(ii)}{\leq}& \sigma_2(\pmb{A})\Vert \pmb{x}\Vert_2\Vert \pmb{y}\Vert_2\leq \frac{C_0}{2}\left(\sqrt{d_1}+\sqrt{d_2}\right) \Vert \pmb{x}\Vert_2\Vert \pmb{y}\Vert_2,
\end{align*} 
where the inequality $(i)$ holds since the matrix
$\pmb{A}$ satisfies the assumptions ${\rm \pmb{\mathscr{G}}_1} $, ${\rm \pmb{\mathscr{G}}_3} $
such that we have $\pmb{x}' \perp \pmb{1}$; the inequality $(ii)$ comes from $\Vert \pmb{x}'\Vert_2 \leq \Vert \pmb{x}\Vert_2$.\\
\indent The second term is upper bounded by
\begin{eqnarray*}
\left\vert \pmb{x}_0\sum_{j=1}^{n_2} \text{deg}(j)\pmb{y}_j \right\vert &\leq& \vert \pmb{x}_0\vert \cdot \max_{j\in[n_2]} \text{deg}(j) \Vert \pmb{y}\Vert_1 \\
& \leq& \frac{d_2}{n_1}\Vert \pmb{x}\Vert_1\Vert \pmb{y}\Vert_1 = \frac{1}{p}\Vert \pmb{x}\Vert_1\Vert \pmb{y}\Vert_1.
\end{eqnarray*} 
Here, we have $\vert \pmb{x}_0\vert = \left\vert  \frac{\sum_{i=1}^{n_1}\pmb{x}_i}{n_1}\right\vert \leq \frac{\Vert \pmb{x}\Vert_1}{n_1}$ and $\text{deg}(j) =d_2$ by $(d_1,d_2)$-biregular graph.\\
\indent Combining above two upper bounds, we complete the proof.
\end{proof}

Based on the definition of $\mathcal{P}_{\underline{\Omega}}$, we can get the followings.
\begin{lemma} \label{l_7}
For all $\pmb{H} \in \mathbb{R}^{(n_1+n_2)\times (n_1+n_2)}$ such that $\Vert \pmb{H}\Vert_{2,\infty} \leq 4\sqrt{\frac{\mu r \sigma_1^*}{n_1 \wedge n_2}}$, we have
\begin{eqnarray*}
\frac{1}{p}\Vert \mathcal{P}_{\underline{\Omega}}(\pmb{H}\pmb{H}^\top)\Vert_F^2 \leq 2\Vert \pmb{H}\Vert_F^4 + \frac{16C_0\mu r \kappa(\sqrt{n_1}+\sqrt{n_2})}{\sqrt{p}(n_1\wedge n_2)}\sigma_r^*\Vert \pmb{H}\Vert_F^2.
\end{eqnarray*} 
\end{lemma}
\begin{proof}
By the definition of $\mathcal{P}_{\underline{\Omega}}$, we have
\begin{align*}
\frac{1}{p}\Vert \mathcal{P}_{\underline{\Omega}}(\pmb{H}\pmb{H}^\tp)\Vert_F^2 
 =& \frac{2}{p}\Vert \mathcal{P}_{\Omega}(\pmb{H}_{\pmb{U}}{(\pmb{H}_{\pmb{V}})}^\tp)\Vert_F^2 = \frac{2}{p} \sum_{(i, j)\in\Omega}\langle (\pmb{H}_{\pmb{U}})_i, (\pmb{H}_{\pmb{V}})_j\rangle^2\\
\leq& \frac{2}{p}\sum_{(i, j)\in\Omega} \Vert (\pmb{H}_{\pmb{U}})_i\Vert_2^2\cdot\Vert (\pmb{H}_{\pmb{V}})_j\Vert_2^2.
\end{align*} 
Based on Lemma \ref{l_6}, by setting $\pmb{x}_i= \Vert (\pmb{H}_{\pmb{U}})_i\Vert_2^2$ and $\pmb{y}_j= \Vert (\pmb{H}_{\pmb{V}})_j\Vert_2^2$, we can get 
\begin{align*}
&\frac{1}{p}\sum_{(i, j)\in\Omega} \Vert (\pmb{H}_{\pmb{U}})_i\Vert_2^2\cdot\Vert (\pmb{H}_{\pmb{V}})_j\Vert_2^2 \\
\leq& \left(\sum_{i=1}^{n_1}\Vert (\pmb{H}_{\pmb{U}})_i\Vert_2^2\right)\cdot \left(\sum_{j=1}^{n_2} \Vert (\pmb{H}_{\pmb{V}})_j\Vert_2^2\right) \\
&+ \frac{C_0}{2}\left(\frac{n_2}{\sqrt{d_1}} + \frac{n_1}{\sqrt{d_2}}\right) \sqrt{\sum_{i=1}^{n_1}\Vert (\pmb{H}_{\pmb{U}})_i\Vert_2^4}\sqrt{\sum_{j=1}^{n_2} \Vert (\pmb{H}_{\pmb{V}})_j\Vert_2^4}\\
\leq& \Vert \pmb{H}\Vert_F^2 \cdot \Vert \pmb{H}\Vert_F^2 + \frac{C_0}{2}\frac{\sqrt{n_1}+\sqrt{n_2}}{\sqrt{p}} \sum_{i=1}^{n_1+n_2}\Vert \pmb{H}_i\Vert_2^4\\
\leq&  \Vert \pmb{H}\Vert_F^4 + \frac{C_0}{2}\frac{\sqrt{n_1}+\sqrt{n_2}}{\sqrt{p}} \Vert \pmb{H}\Vert^2_{2,\infty}\Vert \pmb{H}\Vert_F^2.
\end{align*} 
Under the condition $\Vert \pmb{H}\Vert_{2,\infty} \leq 4\sqrt{\frac{\mu r \sigma_1^*}{n_1 \wedge n_2}}$ in \eqref{H_2infty}, it goes that
\begin{eqnarray*}
\frac{1}{p}\Vert \mathcal{P}_{\underline{\Omega}}(\pmb{H}\pmb{H}^\tp)\Vert_F^2 \leq 2\Vert \pmb{H}\Vert_F^4 + \frac{16C_0\mu r \kappa(\sqrt{n_1}+\sqrt{n_2})}{\sqrt{p}(n_1\wedge n_2)}\sigma_r^*\Vert \pmb{H}\Vert_F^2.
\end{eqnarray*} 
\end{proof}

\begin{lemma} \label{l_8}
The sample operator $\mathcal{P}_{\underline{\Omega}}$ is defined as before. Uniformly for all matrices $\pmb{A}$, $\pmb{B}$ such that $\pmb{A}\pmb{B}^\tp$ is of size $(n_1+n_2) \times (n_1+n_2)$, then 
\begin{eqnarray*}
\frac{1}{p}\Vert \mathcal{P}_{\underline{\Omega}}(\pmb{A}\pmb{B}^\tp)\Vert_F^2 \leq (n_1 \vee n_2) \min \{ \Vert \pmb{A}\Vert_F^2\Vert\pmb{B}\Vert_{2,\infty}^2, \Vert \pmb{B}\Vert_F^2\Vert\pmb{A}\Vert_{2,\infty}^2\}.
\end{eqnarray*} 
\end{lemma}
\begin{proof}
Define that $\underline{\Omega}_i$ is the set of entries sampled in the $i$-th row of $\underline{\Omega}$. By the property of $(d_1,d_2)$-biregular graph, we have 
\begin{align*}
\frac{1}{p}\Vert \mathcal{P}_{\underline{\Omega}}(\pmb{A}\pmb{B}^\tp)\Vert_F^2 
=& \frac{1}{p} \sum_{i=1}^{n_1+n_2}\sum_{j\in \underline{\Omega}_i}\langle \pmb{A}_i, \pmb{B}_j\rangle^2
\leq  \frac{1}{p} \sum_{i=1}^{n_1+n_2}\Vert \pmb{A}_i\Vert_2^2\cdot \sum_{j\in \underline{\Omega}_i}\Vert \pmb{B}_j\Vert_2^2\\
\overset{(i)}{\leq}&  \frac{1}{p} \sum_{i=1}^{n_1}\Vert \pmb{A}_i\Vert_2^2\cdot d_1\Vert \pmb{B}\Vert_{2,\infty}^2 + \frac{1}{p} \sum_{i=n_1+1}^{n_1+n_2}\Vert \pmb{A}_i\Vert_2^2\cdot d_2\Vert \pmb{B}\Vert_{2,\infty}^2\\
\leq&(n_1 \vee n_2)\Vert \pmb{A}\Vert_F^2\Vert\pmb{B}\Vert_{2,\infty}^2,
\end{align*} 
where the inequality $(i)$ comes from $\vert \underline{\Omega}_i\vert = d_1$ if $1\leq i\leq n_1$ while $\vert \underline{\Omega}_i\vert = d_2$ if $n_1+1\leq i\leq n_1+n_2$.\\
\indent Similarly we can prove $\frac{1}{p}\Vert \mathcal{P}_{\underline{\Omega}}(\pmb{A}\pmb{B}^\tp)\Vert_F^2 \leq (n_1 \vee n_2) \Vert \pmb{B}\Vert_F^2\Vert\pmb{A}\Vert_{2,\infty}^2$.
\end{proof}

 The following are also some useful lemmas.
\begin{lemma} \cite[Lemma 4]{BV2020} \label{l_10}
Let $\pmb{E}$ be the all $1$'s matrix. 
Then, 
\bea
\left\Vert \frac{1}{p}\pmb{G}-\pmb{E} \right\Vert \leq \frac{C_0\sqrt{n_1n_2}}{\sqrt{d_1\wedge d_2}}.
\eea

\end{lemma}

\begin{lemma} \cite[Theorem 4.1]{T1} \label{l_1}
Let $\pmb{G}$ satisfy the assumptions ${\rm \pmb{\mathscr{G}}_1} $, ${\rm \pmb{\mathscr{G}}_2} $ and ${\rm \pmb{\mathscr{G}}_3} $. Then,
\begin{eqnarray*}
\left\Vert \frac{1}{p}\mathcal{P}_{\Omega}(\pX)-\pX \right\Vert \leq \frac{C_0\mu r}{\sqrt{d_1\wedge d_2}}\Vert \pX\Vert,
\end{eqnarray*}
where $C_0$ is the constant in the assumption ${\rm \pmb{\mathscr{G}}_2} $. Moreover, if $\pmb{U}^0\pmb{\Sigma}^0(\pmb{V}^0)^\tp$ is the top-$r$ SVD of $\frac{1}{p}\mathcal{P}_{\Omega}(\pM^*)$, we have $\Vert \pmb{U}^0\pmb{\Sigma}^0(\pmb{V}^0)^\tp-\pmb{M}^*\Vert \leq \frac{2C_0\mu r}{\sqrt{d_1\wedge d_2}}\Vert \pmb{M}^*\Vert$.
\end{lemma}

\begin{lemma} \cite[Lemma 4]{T2} \label{l_2}
Let $\overline{\pmb{Z}}{}^0$, $\pmb{Z}^* \in \mathbb{R}^{(n_1+n_2)\times r}$ be defined as before. We have
\begin{eqnarray*}
\Vert \overline{\pmb{Z}}{}^0(\overline{\pmb{Z}}{}^0)^\tp - \pmb{Z}^*(\pmb{Z}^*)^\tp\Vert_F \leq 2 \Vert \pmb{U}^0\pmb{\Sigma}^0 (\pmb{V}^0)^\tp - \pmb{U}^*\pmb{\Sigma}^* (\pmb{V}^*)^\tp\Vert_F.
\end{eqnarray*}
\end{lemma}

The following two lemmas show that the projections $\mathcal{P}_{\mathcal{C}_1}$ and $\mathcal{P}_{\mathcal{C}_2}$ are non-expansive and incoherent.
\begin{lemma} \cite[Lemma 11]{T2} \label{l_3}
Let $\pmb{y}\in\mathbb{R}^n$ be a vector such that $\Vert \pmb{y}\Vert_2\leq a$, for any $x \in \mathbb{R}^n$. Then 
\begin{eqnarray*}
\Vert \mathcal{P}_{\Vert \cdot \Vert_2 \leq a}(\pmb{x})-\pmb{y}\Vert_2^2 \leq \Vert \pmb{x}-\pmb{y}\Vert_2^2.
\end{eqnarray*}
\end{lemma} 

\begin{lemma} \cite[Lemma 19]{TMC2021} \label{non_expan}
Suppose that $\pM^*$ is $\mu$-incoherent, $\pZ = [\pX^{\tp}\ \pY^{\tp}]^{\tp}$ and $\dist_*(\pZ, \pZ^*) \leq \alpha \sigma_r^*$ for some $\alpha < 1$. Set $B \geq (1+\alpha)\sqrt{\mu r} \sigma_1(\pM^*)$, then $\mathcal{P}_{\mathcal{C}_2}(\pZ)$ satisfies the non-expansiveness
\bea
\dist_*(\mathcal{P}_{\mathcal{C}_2}(\pZ), \pZ^*) \leq \dist_*(\pZ, \pZ^*),
\eea 
and the incoherence condition
\bea
\sqrt{n_1} \| \pX\pY^{\tp} \|_{2,\infty} \vee \sqrt{n_1} \| \pY\pX^{\tp} \|_{2,\infty}  \leq B.
\eea
\end{lemma}



\subsection{Proof of Local Curvature and Smoothness Condition} \label{prf_prop34}
Before we give the proofs of  Proposition \ref{prop_curvature} and \ref{prop_smoothness}, we give some upper bounds, which can be found in \cite{TMC2021}.
\begin{lemma} \cite[Claim 6]{TMC2021} \label{lem:upperbound}
Assume that $\dist_*(\pZ,\pZ^*) \leq \alpha \sigma_r^*$ and $\sqrt{n_1}\| \pX\pY^\tp \|_{2,\infty} \vee \sqrt{n_2}\| \pY\pX^\tp \|_{2,\infty} \leq (1+\alpha)\sqrt{\mu r}\sigma_1^*$ where $\pZ = [\pX^\tp \  \pY^\tp]^\tp$, $\pZ^* = [(\pX^*)^\tp \  (\pY^*)^\tp]^\tp$ and $\alpha <1$,  one has
\bea
&&\| (\oX\ \oQ -\pX^*)(\pSig^*)^{-{1\over2}} \| \vee \| (\oY (\oQ)^{-\tp}-\pY^*) (\pSig^*)^{-{1\over2}} \| \leq \alpha; \\
&&\| \pY(\pY^\tp\pY)^{-1}(\pSig^*)^{1\over2} \| \leq \frac{1}{1-\alpha}  ;\\
&&\| (\pSig^*)^{1\over2}(\pY^\tp\pY)^{-1}(\pSig^*)^{1\over2} \| \leq \frac{1}{(1-\alpha)^2}  ;\\
&&  \sqrt{n_1}\| \pX(\pSig^*)^{\frac{1}{2}} \|_{2,\infty} \vee \sqrt{n_2}\| \pY(\pSig^*)^{\frac{1}{2}} \|_{2,\infty} \leq \frac{1+\alpha}{1-\alpha}\sqrt{\mu r}\sigma_1^*;\\
&&  \sqrt{n_1}\| \pX(\pSig^*)^{-\frac{1}{2}} \|_{2,\infty} \vee \sqrt{n_2}\| \pY(\pSig^*)^{-\frac{1}{2}} \|_{2,\infty}  \leq \frac{(1+\alpha)\kappa\sqrt{\mu r}}{1-\alpha};\\
&&  \sqrt{n_1}\| (\oX\ \oQ -\pX^*)(\pSig^*)^{\frac{1}{2}} \|_{2,\infty} \vee \sqrt{n_2}\| [ \oY (\oQ)^{-\tp}-\pY^*](\pSig^*)^{\frac{1}{2}} \|_{2,\infty}  \leq \frac{2}{1-\alpha}\sqrt{\mu r}\sigma_1^*.
\eea
\end{lemma}

\begin{proof}[Proof of Proposition \ref{prop_curvature}]
For convenience, we denote $\pX = \oX{\ } \oQ$, $\pY= \oY (\oQ)^{-\tp}$ and $\pDeX = \pX - \pX^*$, $\pDeY = \pY - \pY^*$. We then have 
\begin{align*}
&\la (\oX\ \oQ - \pX^*)(\pSig^*)^{1\over2}, \nabla_{\pX}\mL_2(\oX, \oY)\left[(\oY)^\tp\oY\right]^{-1}\oQ(\pSig^*)^{1\over2} \ra  \\
=& \la \pDeX(\pSig^*)^{1\over2}, \frac{1}{p}\mP_{\Omega}(\pX\pY^\tp-\pM^*)\pY(\pY^\tp\pY)^{-1}(\pSig^*)^{1\over2} \ra \\
=& \la \pDeX (\pSig^*)^{1\over2}, (\frac{1}{p}\mP_{\Omega}-\mI)(\pX\pY^\tp-\pM^*)\pY(\pY^\tp\pY)^{-1}(\pSig^*)^{1\over2} \ra\\
&+\la \pDeX(\pSig^*)^{1\over2}, (\pX\pY^\tp-\pM^*)\pY(\pY^\tp\pY)^{-1}(\pSig^*)^{1\over2} \ra \\
=:& A_1 +A_2.
\end{align*}
Since 
\be
\pX\pY^\tp-\pM^*  =  \pDeX(\pY^*)^\tp + \pX\pDeY^\tp =  \pDeX(\pY^*)^\tp + \pX\pDeY^\tp, \label{qw}
\ee
 we break the term $A_1$ into three parts:
\begin{align*}
A_1 =& A_{1,1} + A_{1,2} + A_{1,3}\\
=&  \la \pDeX(\pSig^*)^{1\over2}, (\frac{1}{p}\mP_{\Omega}-\mI)( \pDeX(\pY^*)^\tp)\pDeY(\pY^\tp\pY)^{-1}(\pSig^*)^{1\over2} \ra \\
&+ \la \pDeX(\pSig^*)^{1\over2}, (\frac{1}{p}\mP_{\Omega}-\mI)( \pDeX(\pY^*)^\tp)\pY^*(\pY^\tp\pY)^{-1}(\pSig^*)^{1\over2} \ra \\
&+ \la \pDeX(\pSig^*)^{1\over2}, (\frac{1}{p}\mP_{\Omega}-\mI)(\pX\pDeY^\tp)\pY(\pY^\tp\pY)^{-1}(\pSig^*)^{1\over2} \ra.
\end{align*}

For the first one, we use the Lemma 12 in \cite{CLL2020} and the inequality $\Vert \pX\pY \Vert_F \leq \|\pX\| \cdot \|\pY\|_F$  to get
\begin{align*}
|A_{1,1}| 
\leq& \left\Vert \frac{1}{p}\pmb{G} - \pmb{E} \right\Vert \cdot  \Vert \pDeX(\pSig^*)^{1\over2} \Vert_{2,\infty} \Vert \pDeX(\pSig^*)^{1\over2} \Vert_F   \\
&\cdot \Vert\pY^*(\pSig^*)^{-{1\over2}} \Vert_{2,\infty} \Vert \pDeY(\pY^\tp\pY)^{-1}(\pSig^*)^{1\over2} \Vert_F\\
\leq& \frac{C_0\sqrt{n_1n_2}}{\sqrt{d_1\wedge d_2}} \Vert \pDeX(\pSig^*)^{1\over2} \Vert_{2,\infty} \Vert \pDeX(\pSig^*)^{1\over2} \Vert_F \Vert\pY^*(\pSig^*)^{-{1\over2}} \Vert_{2,\infty} \\
&\cdot  \Vert \pDeY(\pSig^*)^{-{1\over2}} \Vert_F \Vert (\pSig^*)^{1\over2}(\pY^\tp\pY)^{-1}(\pSig^*)^{1\over2} \Vert, 
\end{align*}
where the last inequality holds from $\Vert \frac{1}{p}\pmb{G} - \pmb{E} \Vert \leq \frac{C_0\sqrt{n_1 n_2}}{\sqrt{d_1\wedge d_2}}$ in Lemma \ref{l_10}.
Lemma \ref{lem:upperbound} and Assumption ${\rm \pmb{\mathscr{A}}_1}$ yields that 
\bea
|A_{1,1}| \leq \frac{2}{(1-\alpha)^3} \frac{C_0\mu r \kappa}{\sqrt{d_1\wedge d_2}} \Vert \pDeX(\pSig^*)^{{1\over2}} \Vert_F \Vert \pDeY(\pSig^*)^{{1\over2}} \Vert_F,
\eea
where we use the result $\Vert \pDeY(\pSig^*)^{-{1\over2}} \Vert_F \leq \frac{1}{\sigma_r^*}\Vert \pDeY(\pSig^*)^{{1\over2}} \Vert_F$.

For the second part, we use Lemma \ref{l_5} to bound 
\begin{align*}
|A_{1,2}| 
\leq& \tilde{\delta} \| \pDeX (\pY^*)^\tp \|_F \| \pDeX \pSig^* (\pY^\tp\pY)^{-1} (\pY^*)^\tp \|_F \\
\leq& \tilde{\delta} \| \pDeX (\pSig^*)^{1\over2} (\pV^*)^\tp \|_F   \| (\pSig^*)^{1\over2} (\pY^\tp\pY)^{-1} (\pSig^*)^{1\over2} (\pV^*)^\tp \|  
 \| \pDeX (\pSig^*)^{1\over2} \|_F \\
=& \tilde{\delta}\| (\pSig^*)^{1\over2} (\pY^\tp\pY)^{-1} (\pSig^*)^{1\over2} \|   \| \pDeX (\pSig^*)^{1\over2} \|_F^2,
\end{align*}
where the second inequality comes from $\pY^* = \pV^*(\pSig^*)^{1\over2}$ and $\|\pX\pY\|_F \leq \|\pX\| \|\pY\|_F$.
Lemma \ref{lem:upperbound} yields that
\bea
|A_{1,2}| \leq \frac{1}{(1-\alpha)^2} \sqrt{2\Big(\delta_d^2 + \frac{C_0^2\mu^2 r^2}{d_1\wedge d_2}\Big)} \| \pDeX (\pSig^*)^{1\over2} \|_F^2.
\eea

Similar to $|A_{1,1}|$, we can bound $A_{1,3}$ as
\begin{align*}
|A_{1,3}| 
\leq&  \frac{C_0\sqrt{n_1n_2}}{\sqrt{d_1\wedge d_2}} \cdot \| \pX (\pSig^*)^{-{1\over2}} \|_{2,\infty} \|\pDeX (\pSig^*)^{1\over2}\|_F \| \pY(\pY^\tp\pY)^{-1}(\pSig^*)^{1\over2} \|_{2,\infty} \|\pDeY(\pSig^*)^{1\over2}\|_F \\
\leq& \frac{C_0\sqrt{n_1n_2}}{\sqrt{d_1\wedge d_2}} \cdot \| \pX (\pSig^*)^{-{1\over2}} \|_{2,\infty} \|\pDeX(\pSig^*)^{1\over2}\|_F\\
&\cdot  \| (\pSig^*)^{1\over2}(\pY^\tp\pY)^{-1}(\pSig^*)^{1\over2} \|   \| \pY(\pSig^*)^{-{1\over2}}\|_{2,\infty}  \|\pDeY(\pSig^*)^{1\over2}\|_F \\
\leq& \frac{(1+\alpha)^2}{(1-\alpha)^4} \frac{C_0\mu r \kappa^2}{\sqrt{d_1\wedge d_2}} \|\pDeX(\pSig^*)^{1\over2}\|_F\|\pDeY(\pSig^*)^{1\over2}\|_F.
\end{align*}
Combining with these three parts and setting $C_1 = \frac{1}{(1-\alpha)^2} \sqrt{2\Big(\delta_d^2 + \frac{C_0^2\mu^2 r^2}{d_1\wedge d_2}\Big)}$ and $C_2 = \left[ 2 + \frac{(1+\alpha)^2\kappa}{1-\alpha}\right] \frac{C_0\mu r  \kappa}{(1-\alpha)^3\sqrt{d_1\wedge d_2}}$, we can see that 
\bea
A_1 &\geq& -C_1\|\pDeX(\pSig^*)^{1\over2}\|_F^2 -C_2 \|\pDeX(\pSig^*)^{1\over2}\|_F\|\pDeY(\pSig^*)^{1\over2}\|_F \\
&\geq& -(C_1 + \frac{C_2}{2}) \|\pDeX(\pSig^*)^{1\over2}\|_F^2 - \frac{C_2}{2} \|\pDeY(\pSig^*)^{1\over2}\|_F^2,
\eea
where the last inequality comes from $ab \leq\frac{a^2+b^2}{2}$.

Notice that 
\bea
A_2 = \| \pDeX(\pSig^*)^{1\over2} \|_F^2 + \la \pDeX(\pSig^*)^{1\over2}, \pX^* \pDeY^{\tp}\pY(\pY^{\tp}\pY)^{-1}(\pSig^*)^{1\over2} \ra,
\eea
because of  the equation \eqref{qw}. We then turn to estimate the latter one and we denote it as $A_{2,2}$ for short.

Utilizing the equalities $\pX^* = \pX - \pDeX$ and $\pX^\tp\pDeX \pSig^* = \pSig^* \pDeY^\tp\pY$ in Lemma 23 of \cite{TMC2021}, we can see
\bea
A_{2,2} &=& \la \pDeX(\pSig^*)^{1\over2}, (\pX - \pDeX) \pDeY^{\tp}\pY(\pY^{\tp}\pY)^{-1}(\pSig^*)^{1\over2} \ra \\
&=& \la \pX^\tp\pDeX\pSig^*,   \pDeY^{\tp}\pY(\pY^{\tp}\pY)^{-1} \ra 
- \la \pDeX\pSig^*,  \pDeX \pDeY^{\tp}\pY(\pY^{\tp}\pY)^{-1} \ra \\
&=& \la \pY^\tp \pDeY \pSig^* \pDeY^\tp\pY,   (\pY^{\tp}\pY)^{-1} \ra 
- \la \pDeX\pSig^*,  \pDeX \pDeY^{\tp}\pY(\pY^{\tp}\pY)^{-1} \ra.
\eea
Since $ \la \pY^\tp \pDeY \pSig^* \pDeY^\tp\pY,   (\pY^{\tp}\pY)^{-1} \ra$ is the inner product of two PSD matrices, we have $ \la \pY^\tp \pDeY \pSig^* \pDeY^\tp\pY,   (\pY^{\tp}\pY)^{-1} \ra \geq 0$. Then it arrives at
\bea
A_{2,2} \geq - \la \pDeX\pSig^*,  \pDeX \pDeY^{\tp}\pY(\pY^{\tp}\pY)^{-1} \ra.
\eea
Lemma \ref{lem:upperbound} yields that
\begin{align*}
\la \pDeX\pSig^*,  \pDeX \pDeY^{\tp}\pY(\pY^{\tp}\pY)^{-1} \ra 
 =& \la  (\pSig^*)^{1\over2}\pDeX^\tp \pDeX (\pSig^*)^{1\over2}, (\pSig^*)^{-{1\over2}} \pDeY^{\tp}\pY(\pY^{\tp}\pY)^{-1} (\pSig^*)^{1\over2} \ra \\
\leq& \| (\pSig^*)^{-{1\over2}} \pDeY^{\tp}\pY(\pY^{\tp}\pY)^{-1} (\pSig^*)^{1\over2} \| \cdot  \| (\pSig^*)^{1\over2}\pDeX^\tp \pDeX (\pSig^*)^{1\over2} \|_F \\
\leq& \|  \pDeY (\pSig^*)^{-{1\over2}} \| \cdot \|\pY(\pY^{\tp}\pY)^{-1} (\pSig^*)^{1\over2} \| \cdot  \| \pDeX (\pSig^*)^{1\over2} \|_F^2 \\
\leq& \frac{\alpha}{1-\alpha} \| \pDeX (\pSig^*)^{1\over2} \|_F^2,
\end{align*}
which implies
\bea
A_{2,2} \geq -\frac{\alpha}{1-\alpha} \| \pDeX (\pSig^*)^{1\over2} \|_F^2.
\eea
Combining with the bounds of $A_1$ and $A_2$, we have
\begin{align*}
&\la (\oX\ \oQ - \pX^*)(\pSig^*)^{1\over2}, \nabla_{\pX}\mL_2(\oX, \oY)[(\oY)^\tp\oY]^{-1}\oQ(\pSig^*)^{1\over2} \ra \\
\geq& \Big(1-C_1-\frac{C_2}{2} - \frac{\alpha}{1-\alpha}\Big) \| \pDeX (\pSig^*)^{1\over2} \|_F^2 
- \frac{C_2}{2} \| \pDeY (\pSig^*)^{1\over2} \|_F^2.
\end{align*}
When setting $\alpha =0.1$, $\delta_d = \frac{1}{64}$ and $d_1\wedge d_2 \geq 100^2C_0^2\mu^2 r^2 \kappa^4$, we can obtain $C_1 \leq 0.033$, $C_2 \leq 0.046$. Furthermore, we get
\bea
\la (\oX\ \oQ - \pX^*)(\pSig^*)^{1\over2}, \nabla_{\pX}\mL_2(\oX, \oY)[(\oY)^\tp\oY]^{-1}\oQ(\pSig^*)^{1\over2} \ra \\
\geq 0.833 \| \pDeX (\pSig^*)^{1\over2} \|_F^2 - 0.023\| \pDeY (\pSig^*)^{1\over2} \|_F^2.
\eea
Based on the symmetry property of $\pX$ and $\pY$, we also get
\bea
\la (\oY(\oQ)^{-\tp} - \pY^*)(\pSig^*)^{1\over2}, \nabla_{\pY}\mL_2(\oX, \oY)[(\oX)^\tp\oX]^{-1}(\oQ)^{-\tp}(\pSig^*)^{1\over2} \ra\\
 \geq 0.833 \| \pDeY (\pSig^*)^{1\over2} \|_F^2 - 0.023\| \pDeX (\pSig^*)^{1\over2} \|_F^2.
\eea
\end{proof}

\begin{proof}[Proof of Proposition \ref{prop_smoothness}]
We  denote $\pX = \oX{\ } \oQ$, $\pY= \oY(\oQ)^{-\tp}$ and $\pDeX = \pX - \pX^*$, $\pDeY = \pY - \pY^*$ for convenience. For $\|\nabla_{\pX}\mL_2(\oX, \oY)[(\oY)^\tp\oY]^{-1}\oQ(\pSig^*)^{1\over2}\|_F^2$, 
we write
\begin{align*}
&\| \frac{1}{p}\mP_{\Omega}(\pX\pY^\tp-\pM^*)\pY(\pY^\tp\pY)^{-1}(\pSig^*)^{1\over2} \|_F  \\
\leq& \| (\frac{1}{p}\mP_{\Omega}-\mI)(\pX\pY^\tp-\pM^*)\pY(\pY^\tp\pY)^{-1}(\pSig^*)^{1\over2} \|_F 
+ \| (\pX\pY^\tp-\pM^*)\pY(\pY^\tp\pY)^{-1}(\pSig^*)^{1\over2} \|_F\\
=:& B_1+B_2.
\end{align*}
Recalling that $\| \pmb{A} \|_F = \max_{\|\pmb{B}\|_F=1} \langle \pmb{A}, \pmb{B}\rangle$, there exists a matrix $\widetilde{\pmb{R}} \in \mbR^{n_1\times r}$ with $\| \widetilde{\pmb{R}} \|_F=1$ satisfying
\begin{align*}
& \| (\frac{1}{p}\mP_{\Omega}-\mI)(\pDeX(\pY^*)^\tp + \pX\pDeY^\tp)\pY(\pY^\tp\pY)^{-1}(\pSig^*)^{1\over2} \|_F \\
\leq& |\la (\frac{1}{p}\mP_{\Omega}-\mI)(\pDeX(\pY^*)^\tp + \pX\pDeY^\tp)\pY(\pY^\tp\pY)^{-1}(\pSig^*)^{1\over2} , \widetilde{\pmb{R}} \ra| \\
\leq& |\la (\frac{1}{p}\mP_{\Omega}-\mI)(\pDeX(\pY^*)^\tp), \widetilde{\pmb{R}}(\pSig^*)^{1\over2}(\pY^\tp\pY)^{-1}(\pY^*)^\tp \ra| \\
&+ |\la (\frac{1}{p}\mP_{\Omega}-\mI)(\pDeX(\pY^*)^\tp), \widetilde{\pmb{R}}(\pSig^*)^{1\over2}(\pY^\tp\pY)^{-1}\pDeY^\tp \ra| \\
&+ |\la (\frac{1}{p}\mP_{\Omega}-\mI)(\pX\pDeY^\tp), \widetilde{\pmb{R}}(\pSig^*)^{1\over2}(\pY^\tp\pY)^{-1}\pY^\tp \ra| \\
=:& B_{1,1} + B_{1,2} + B_{1,3},
\end{align*}
where we use the equation \eqref{qw}.

Invoke Lemmas \ref{l_5} and \ref{lem:upperbound} to obtain
\begin{align*}
B_{1,1} \leq& \tilde{\delta} \| \pDeX(\pY^*)^\tp \|_F \cdot \| \widetilde{\pmb{R}}(\pSig^*)^{1\over2}(\pY^\tp\pY)^{-1}(\pY^*)^\tp \|_F \\
\leq& \tilde{\delta}  \| \pDeX(\pSig^*)^{1\over2} \|_F  \| \widetilde{\pmb{R}}\|_F
 \| (\pSig^*)^{1\over2}(\pY^\tp\pY)^{-1}(\pSig^*)^{1\over2}\|  \|\pV^* \|_F ^2\\
\leq& \frac{1}{(1-\alpha)^2} \sqrt{2(\delta_d^2 + \frac{C_0^2\mu^2 r^2}{d_1\wedge d_2})} \| \pDeX(\pSig^*)^{1\over2} \|_F.
\end{align*}
Similar to the estimation of $A_{1,1}$, we can get
\begin{align*}
B_{1,2} \leq& \frac{C_0\sqrt{n_1n_2}}{\sqrt{d_1\wedge d_2}}  \| \pDeX (\pSig^*)^{1\over2}\|_{2,\infty}  \| \widetilde{\pmb{R}}\|_F \| \pY^*(\pSig^*)^{-{1\over2}}\|_{2,\infty} 
 \| (\pSig^*)^{1\over2}(\pY^\tp\pY)^{-1}(\pSig^*)^{1\over2} \| \| \pDeY(\pSig^*)^{-{1\over2}}\|_F \\
\leq& \frac{2}{(1-\alpha)^3} \frac{C_0\mu r \kappa}{\sqrt{d_1\wedge d_2}}  \| \pDeY(\pSig^*)^{1\over2} \|_F
\end{align*}
and 
\begin{align*}
B_{1,3} \leq& \frac{C_0\sqrt{n_1n_2}}{\sqrt{d_1\wedge d_2}}  \| \pX (\pSig^*)^{-{1\over2}}\|_F  \| \widetilde{\pmb{R}}\|_F \| \pDeY(\pSig^*)^{{1\over2}}\|_{2,\infty} 
 \| (\pSig^*)^{1\over2}(\pY^\tp\pY)^{-1} (\pSig^*)^{1\over2} \| \| \pY(\pSig^*)^{-{1\over2}}\|_{2,\infty} \\ 
\leq& \frac{(1+\alpha)^2}{(1-\alpha)^4} \frac{C_0\mu r \kappa^2}{\sqrt{d_1\wedge d_2}}  \| \pDeY(\pSig^*)^{1\over2} \|_F.
\end{align*}
Put all together to get
\begin{align*}
B_1 \leq C_1 \| \pDeX(\pSig^*)^{1\over2} \|_F + C_2\| \pDeY(\pSig^*)^{1\over2} \|_F,
\end{align*}
where the constants are the same ones in the setting of Proposition \ref{prop_curvature}.
Utilizing Lemma 26 in \cite{TMC2021}, we can see that
\begin{align*}
B_2 \leq& \| \pY(\pY^\tp\pY)^{-1}(\pSig^*)^{1\over2} \| \cdot \| \pX\pY^\tp-\pM^* \|_F \\
\leq& \frac{1}{1-\alpha} \left[1+ \frac{1}{2} \left(\| \pDeX(\pSig^*)^{-{1\over2}} \| \vee \| \pDeY(\pSig^*)^{-{1\over2}} \| \right) \right] 
\left[ \| \pDeX (\pSig^*)^{{1\over2}} \|_F + \| \pDeY(\pSig^*)^{{1\over2}} \|_F\right]  \\
\leq& \left(1+\frac{\alpha}{2}\right) \frac{1}{1-\alpha}  \left[\| \pDeX(\pSig^*)^{{1\over2}} \|_F + \| \pDeY(\pSig^*)^{{1\over2}} \|_F\right].
\end{align*}
Thus, we can get
\begin{align*}
\|\nabla_{\pX}\mL_2(\oX, \oY)[(\oY)^\tp\oY]^{-1}\oQ(\pSig^*)^{1\over2}\|_F^2 \leq& 2(B_1^2 +B_2^2) \\
\leq& \left[2C_1(C_1+C_2)+\left(\frac{2+\alpha}{1-\alpha}\right)^2 \right] \| \pDeX(\pSig^*)^{{1\over2}} \|_F^2 \\
&+ \left[2C_2(C_1+C_2)+\left(\frac{2+\alpha}{1-\alpha}\right)^2 \right] \| \pDeY(\pSig^*)^{{1\over2}} \|_F^2. 
\end{align*}
When setting $\alpha = 0.1$, $\delta_d = \frac{1}{64}$ and $d_1\wedge d_2 \geq 100^2C_0^2\mu^2 r^2 \kappa^4$, we can obtain $C_1 \leq 0.033$, $C_2 \leq 0.046$. Furthermore, we get
\begin{align*}
&\|\nabla_{\pX}\mL_2(\oX, \oY)[(\oY)^\tp\oY]^{-1}\oQ(\pSig^*)^{1\over2}\|_F^2\\
 \leq& 2(B_1^2 +B_2^2) \leq 5.5 \left[\| \pDeX(\pSig^*)^{{1\over2}} \|_F^2 + \| \pDeY(\pSig^*)^{{1\over2}} \|_F^2\right].
\end{align*}
Based on the symmetry property of $\pX$ and $\pY$, we also get
\bea
&&\|\nabla_{\pY}\mL_2(\oX, \oY)[(\oX)^\tp\oX]^{-1}(\oQ)^{-\tp}(\pSig^*)^{1\over2}\|_F^2\\
& \leq &2(B_1^2 +B_2^2) \leq 5.5 \left[ \| \pDeX(\pSig^*)^{{1\over2}} \|_F^2 + \| \pDeY(\pSig^*)^{{1\over2}} \|_F^2\right].
\eea
\end{proof}

\subsection{Initialization Behaves Well}
\begin{proof}[Proof of Theorem \ref{thm:initialization2}]
Lemma 24 in \cite{TMC2021} evokes
\bea
\dist_*(\oZ{}^0, \pZ^*) \leq (\sqrt{2}+1)^{1\over 2} \left\Vert \oU{}^0 \oSig{}^0 (\oV{}^0)^\tp - \pM^*  \right\Vert_F.
\eea
Based on Lemma \ref{l_1}, we have 
\begin{align}
\dist_*(\oZ{}^0, \pZ^*) 
  \overset{(i)}{\leq}& (\sqrt{2}+1)^{1\over 2} \sqrt{2r}\Vert \oU{}^0\oSig{}^0 (\oV{}^0)^\tp - \pmb{M}^*\Vert\nonumber\\
\leq& (\sqrt{2}+1)^{1\over 2} \sqrt{2r}\cdot\frac{2C_0\mu r}{\sqrt{d_1\wedge d_2}}\Vert \pmb{M}^*\Vert \nonumber \\
\leq& \frac{2\sqrt{5}C_0\mu r^{1.5}}{\sqrt{d_1\wedge d_2}}\Vert \pmb{M}^*\Vert. \label{prf1_thm:initialization}
\end{align}
Here, the inequality (i) holds because of Cauchy-Schwarz inequality and $\text{rank}(\oU{}^0\oSig{}^0 (\oV{}^0)^\tp - \pmb{M}^*) \leq \text{rank}(\oU{}^0\oSig{}^0 (\oV{}^0)^\tp) + \text{rank}(\pmb{M}^*) \leq 2r$. 
When $d_1 \wedge d_2 \geq 2000C_0^2\mu^2\kappa^2r^3$, we can get
\be
\dist_*(\oZ{}^0, \pZ^*)  \leq  \frac{2\sqrt{5}C_0 \mu \kappa r^{1.5}}{\sqrt{d_1\wedge d_2}}  \cdot\sigma_r^* \leq 0.1 \sigma_r^*, \label{prf3_thm:initialization}
\ee
which completes the proof.

\end{proof}

\subsection{Convergent Analysis}

\begin{proof}[Proof of Theorem \ref{thm_scaledPGD}]
Based on Proposition \ref{prop3} and Lemma \ref{non_expan} with the restriction of stepsize $\eta \leq 0.145$, the $k$-th step of Algorithm \ref{alg:scaledPGD} satisfies
\bea
\dist_{*}^2(\pZ^{k}, \pZ^*) &\leq& (1-1.6\eta + 11\eta^2) \dist_{*}^2(\pZ^{k-1}, \pZ^*) \\
&\leq& (1-1.6\eta + 11\eta^2)^{k} \dist_{*}^2(\pZ^{0}, \pZ^*) \\
&\leq& \frac{1}{100} (1-1.6\eta + 11\eta^2)^{k}  (\sigma_r^*)^2,
\eea
where the last inequality holds based on Theorem \ref{thm:initialization2}. Moreover, we have
\bea
\Vert \pX^{k}(\pY^{k})^\tp - \pM^* \Vert_F \leq \frac{3}{20} (1-1.6\eta + 11\eta^2)^{\frac{k}{2}}  \sigma_r^*.
\eea

\end{proof}

\section{Numerical Experiments} \label{sec4}

\begin{figure}[!h]
\vspace{-0.3cm}  

\setlength{\abovecaptionskip}{0.2cm}   

\setlength{\belowcaptionskip}{0.2cm}   
  \centering
  \includegraphics[scale=0.4]{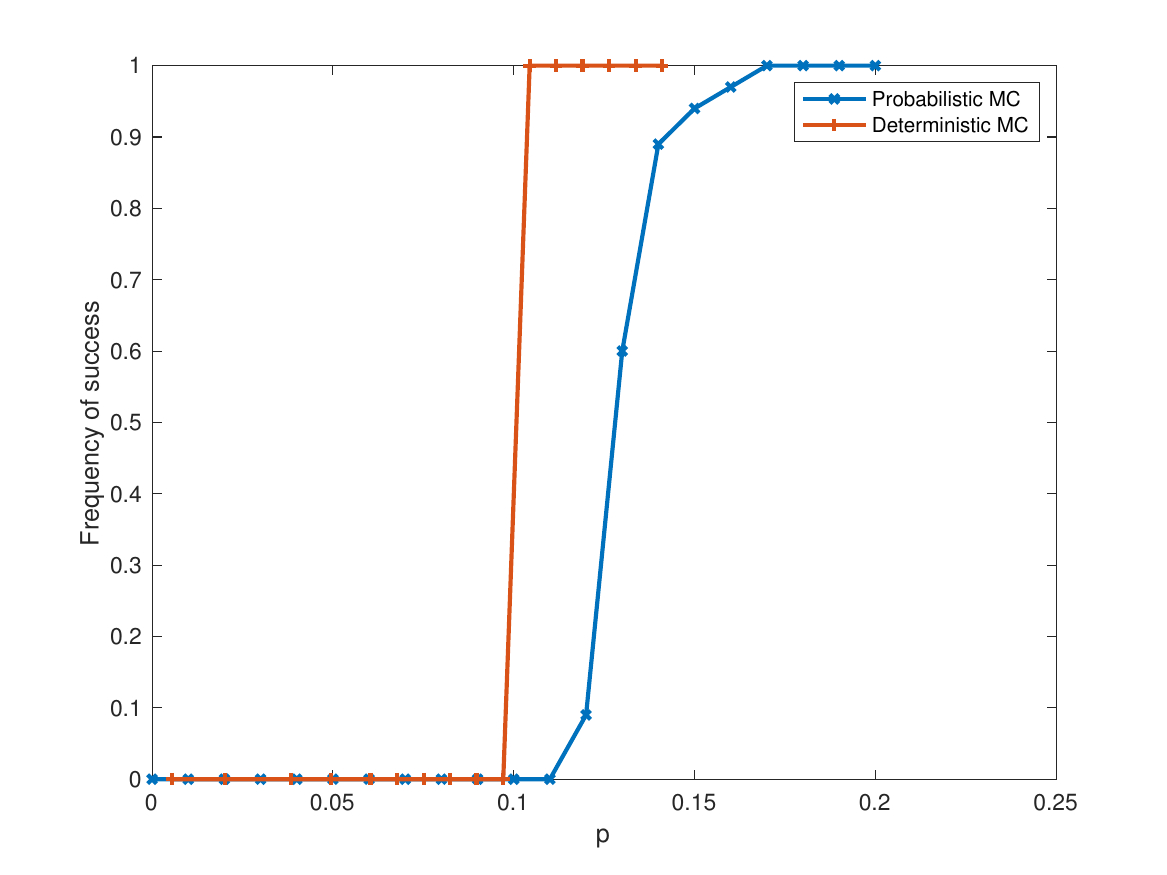}
  \caption{Frequency of success for probabilistic and deterministic matrix completion.}
\label{Fig1}
\end{figure}

We randomly generate a ground-truth $\pM^*$ of size $1092\times 1092$ and rank $3$. It is constructed from two random $1092\times 3$ matrices with i.i.d. normal entries. We compare the frequency of successful recovery for probabilistic and deterministic matrix completion by PGD in Figure \ref{Fig1}. Here, we set $\mu=1.1$ and step-size $\eta$ is tuned. 
And the Ramanujan graph is constructed by LPS construction presented in \cite{lubotzky1988ramanujan} and the MATLAB code for constructing the Ramanujan graph comes from Dr. S. Burnwal and Prof. M. Vidyasagar\footnote{Personal communication.}.
If the relative error $\| \hat{\pM} - \pM^* \|_F/\| \pM^*\|_F$ is less than $10^{-6}$, we say that the recovery is successful in a trial. We average over 100 such trials to determine the success ratio. Figure \ref{Fig1} shows the improvement of deterministic method over probabilistic method for matrix completion. We can see that the deterministic method needs smaller    sample rate $p$ than probabilistic method to achieve success ratio 1 and deterministic method removes the randomness in the recovery.
Thus, we can witness the advantages of deterministic recovery over probabilistic ones.

We use the Inexact Augmented Lagrangian Method (IALM) to solve NNM which is presented in \cite{lin2010augmented}. A brief comparison of the three algorithms is presented in Table \ref{tab1}. We set the tolerated relative error bound $\epsilon < 10^{-4}$ and record the number of iteration together with the running time by averaging 20 trials. One can see that PGD and Scaled PGD are always more efficient than ILAM in the large-scale cases. 
It is worth mentioning that 
we strictly follow the iterates in Algorithm \ref{alg:pgd} to test the running time of PGD. It leads to the major difference of running time between PGD and Scaled PGD since the term $\pmb{D}\pmb{Z}\pmb{Z}^\tp \pmb{D}\pmb{Z}$ in the gradient of $ \mL_1(\pmb{Z})$ in \eqref{lo2plus} by matrix lifting is more computationally expensive while the loss function $\mL_2(\pmb{Z})$ has no regularizer. For further comparison, we use $\left[                 
  \begin{array}{ccc}   
   \pmb{X}(\pX^\tp\pX-\pY^\tp\pY)\\  
    \pmb{Y}(\pY^\tp\pY-\pX^\tp\pX)\\  
  \end{array}
\right]$ instead of $\pmb{D}\pmb{Z}\pmb{Z}^\tp \pmb{D}\pmb{Z}$ in Table \ref{tab2}.

\begin{table}[!ht]

\begin{center}

\caption{Comparison between IALM, PGD and Scaled PGD on the deterministic matrix completion problem under the tolerated relative error bound $\epsilon < 10^{-4}$.}
\label{tab1}
\begin{tabular}{|c|c|c|c|c|c|} \hline
algorithm & $\text{rank}(\pM^*)$ & $p$& \#iter & time (s) & $\frac{\| \hat{\pM} - \pM^* \|_F}{\| \pM^*\|_F}$ \\ \hline
IALM & 2 & 158/2448 & 221 & 1045.43 &8.68e-5 \\ \hline
PGD & 2 & 158/2448 & 140 & 407.89 & 9.96e-5 \\ \hline
Scaled PGD & 2 & 158/2448 & 65 & 9.85 & 9.98e-5 \\ \hline
IALM & 3 & 158/2448 & 211 & 1105.24 &8.87e-5 \\ \hline
PGD & 3 & 158/2448 & 155 & 444.47 & 9.64e-5 \\ \hline
Scaled PGD & 3 & 158/2448 & 72 & 10.66 & 9.79e-5 \\ \hline
IALM & 2 & 298/2448 & 124 & 599.33 &7.69e-5 \\ \hline
PGD & 2 & 298/2448 & 119 & 352.44 & 9.50e-5 \\ \hline
Scaled PGD & 2 & 298/2448 & 56 & 8.89 & 9.35e-5 \\ \hline
IALM & 3 & 298/2448 & 124 & 625.55 &7.62e-5 \\ \hline
PGD & 3 & 298/2448 & 123 & 365.81 & 9.96e-5 \\ \hline
Scaled PGD & 3 & 298/2448 & 58 & 9.19 & 9.35e-5 \\ \hline
\end{tabular}

\end{center}
\end{table}

\begin{table}[!ht]

\begin{center}

\caption{Comparison between IALM, PGD and Scaled PGD on the deterministic matrix completion problem under the tolerated relative error bound $\epsilon < 10^{-4}$.}
\label{tab2}
\begin{tabular}{|c|c|c|c|c|c|} \hline
algorithm & $\text{rank}(\pM^*)$ & $p$& \#iter & time (s) & $\frac{\| \hat{\pM} - \pM^* \|_F}{\| \pM^*\|_F}$ \\ \hline
IALM & 2 & 158/2448 & 221 & 1045.43 &8.68e-5 \\ \hline
PGD & 2 & 158/2448 & 140 & 15.47 & 9.96e-5 \\ \hline
Scaled PGD & 2 & 158/2448 & 65 & 9.85 & 9.98e-5 \\ \hline
IALM & 3 & 158/2448 & 211 & 1105.24 &8.87e-5 \\ \hline
PGD & 3 & 158/2448 & 155 & 16.67 & 9.64e-5 \\ \hline
Scaled PGD & 3 & 158/2448 & 72 & 10.66 & 9.79e-5 \\ \hline
IALM & 2 & 298/2448 & 124 & 599.33 &7.69e-5 \\ \hline
PGD & 2 & 298/2448 & 119 & 13.34 & 9.50e-5 \\ \hline
Scaled PGD & 2 & 298/2448 & 56 & 8.89 & 9.35e-5 \\ \hline
IALM & 3 & 298/2448 & 124 & 625.55 &7.62e-5 \\ \hline
PGD & 3 & 298/2448 & 123 & 13.72 & 9.96e-5 \\ \hline
Scaled PGD & 3 & 298/2448 & 58 & 9.19 & 9.35e-5 \\ \hline
\end{tabular}

\end{center}
\end{table}

\begin{figure}[!t]
\vspace{-0.3cm}  

\setlength{\abovecaptionskip}{0.2cm}   

\setlength{\belowcaptionskip}{0.2cm}   
\centering 

\begin{minipage}[b]{0.45\textwidth} 
\centering 
\includegraphics[width=1\textwidth]{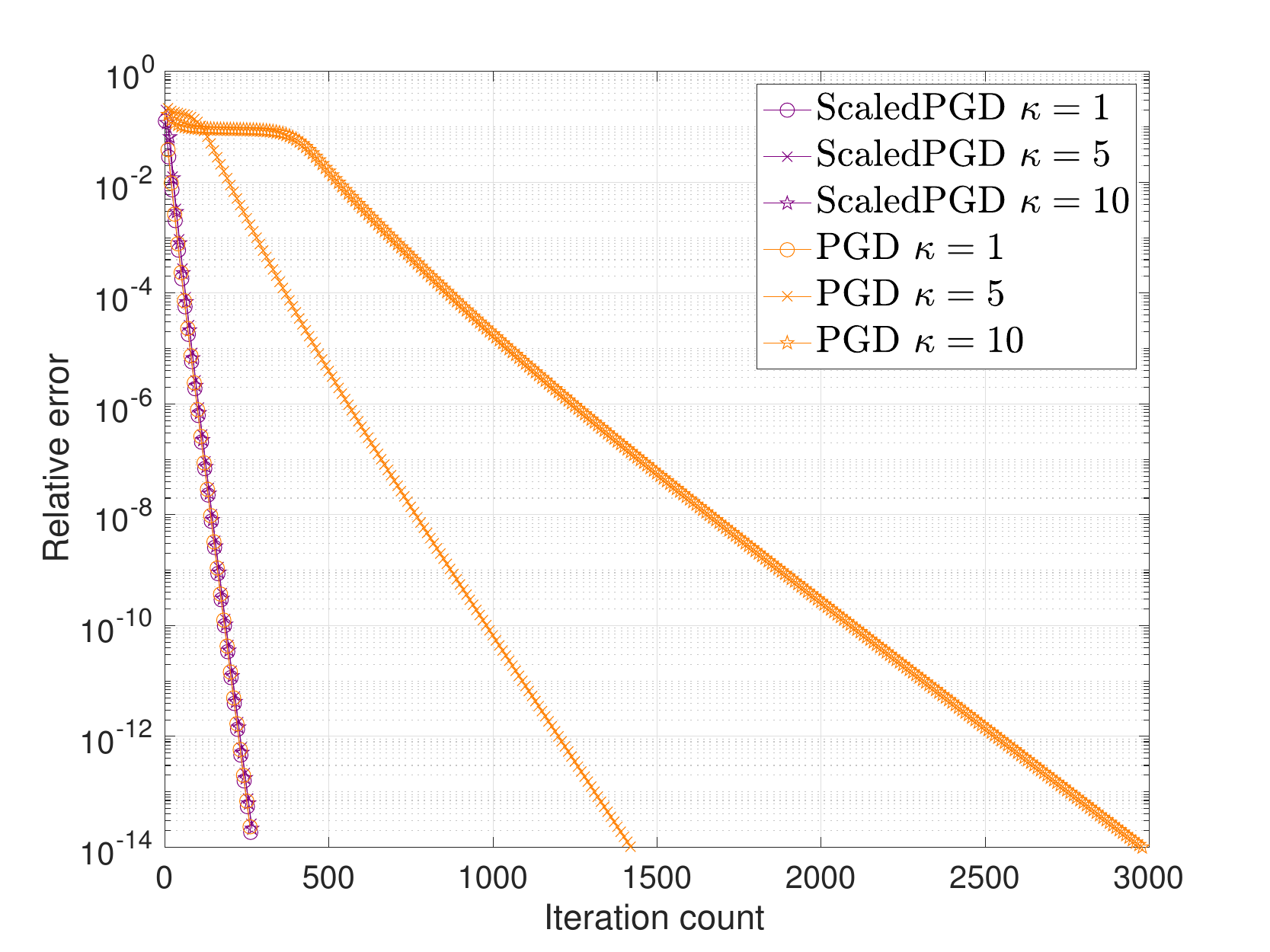} 
$r=3$.
\end{minipage}
\begin{minipage}[b]{0.45\textwidth} 
\centering 
\includegraphics[width=1\textwidth]{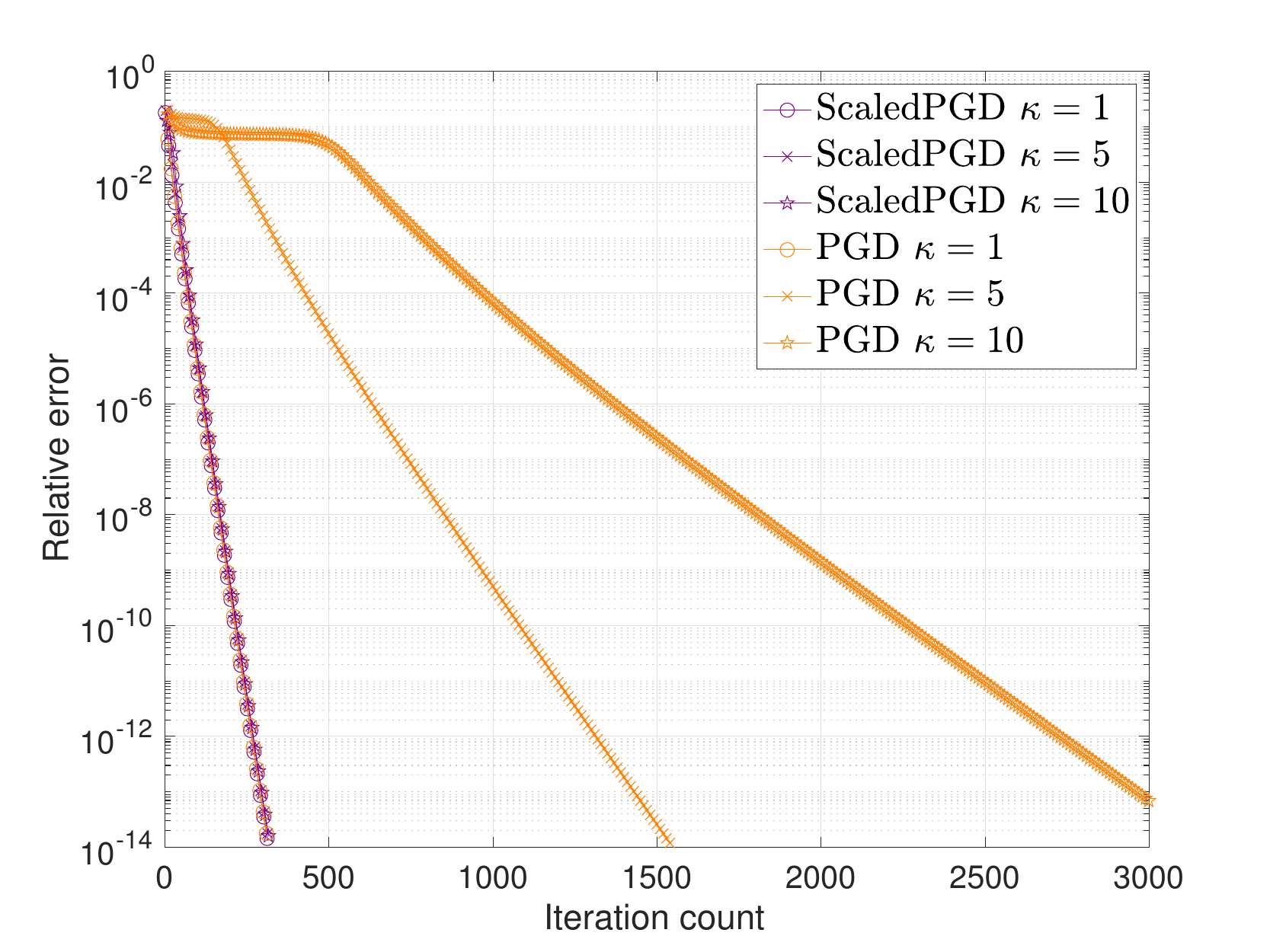}
$r=5$.
\end{minipage}
\caption{The relative errors of PGD and Scaled PGD with $r=3, 5$ under different condition numbers $\kappa = 1,5,10$ for deterministic matrix completion.}
\label{Fig2}
\end{figure}

We next illustrate the convergence rate of PGD and Scaled PGD. Figure \ref{Fig2} plots the tolerated relative errors of two algorithms as iteration count varies. When the condition number $\kappa =1$, the convergence rate of Scaled PGD is same as ones of PGD. And we can see the convergence rate of Scaled PGD is independent of the condition number $\kappa$ while convergence rate of PGD slows down along with $\kappa$ growing. It is consistent with our theoretical results.



\bibliographystyle{IEEEtran}
\bibliography{reference}

\newpage

\section*{Appendices}
\section*{A. Convergent Analysis of PGD} 

For the convenience of the proof, we set $n_1=n_2=n$ and $d_1=d_2=d$ in this part. The proof is similar to \cite{T2}.
Here, we set $T$ to be the subspace of $\mathbb{R}^{n_1\times n_2}$ which is spanned by elements of the form $\pmb{U}^*\pmb{X}^\tp$ and $\pmb{Y}(\pmb{V}^*)^\tp$, that is, 
\bea
T= \{ \pmb{U}^*\pmb{X}^\tp + \pmb{Y}(\pmb{V}^*)^\tp | \pX \in \mathbb{R}^{n_1\times r},  \pY \in \mathbb{R}^{n_2\times r}\}. \label{def_T}
\eea
Based on the definition of $\underline{\Omega}$, we can rewrite the objective function:
\begin{eqnarray}
 \mL_1(\pmb{Z})=\frac{1}{2p}\sum_{l=1}^{2m}\left( \langle \pmb{A}_l, \pmb{Z}\pmb{Z}^\tp \rangle - \pmb{b}_l\right)^2+ \frac{\lambda}{4}\Vert \pmb{Z}^\tp \pmb{D} \pmb{Z}\Vert_F^2, \label{lo3}
\end{eqnarray}
where $m = n_1 d_1 = n_2  d_2$, $\pmb{b}_l = \langle \pmb{A}_l, \pmb{Z}^*(\pmb{Z}^*)^\tp \rangle$ and $\pmb{A}_l$ is a matrix with 1 at one observed entry and 0 elsewhere. Hence, the gradient can be given as
\begin{align}
\nabla  \mL_1(\pmb{Z})=&\frac{1}{p}\sum_{l=1}^{2m}(\langle \pmb{A}_l,\pmb{Z}\pmb{Z}^\tp\rangle -\pmb{b}_l)\cdot(\pmb{A}_l+\pmb{A}_l^\tp)\pmb{Z} 
+\lambda \pmb{D}\pmb{Z}\pmb{Z}^\tp \pmb{D}\pmb{Z} \label{lo4}.
\end{align}

Let $\overline{\pmb{Z}} =\arg\min_{\widetilde{\pmb{Z}} \in \mathcal{S}}\Vert \pmb{Z} - \widetilde{\pmb{Z}} \Vert_F$ and $\pmb{H} = \pmb{Z} - \overline{\pmb{Z}}$. Hence, the gradient (\ref{lo4}) can be written as 
\begin{align}
\nabla  \mL_1(\pmb{Z}
= \frac{1}{p}\sum_{l=1}^{2m}(\langle \pmb{A}_l,\pmb{H}(\overline{\pmb{Z}})^\tp+\overline{\pmb{Z}}\pmb{H}^\tp+\pmb{H}\pmb{H}^\tp\rangle)(\pmb{A}_l+\pmb{A}_l^\tp)(\overline{\pmb{Z}}+\pmb{H})
+\lambda \pmb{D}\pmb{Z}\pmb{Z}^\tp \pmb{D}\pmb{Z}. \label{lo5}
\end{align}


\subsection*{A.1 Initialization Behaves Well}
Firstly, we give a result about our spectral initialization. It can be  proved that in some condition, $\pmb{Z}^0$ is within a small neighborhood of $\pmb{Z}^*$.
\begin{theorem}\label{1}
There exists a constant $C$, so that if the sample size $p \geq \frac{128}{\sqrt{2}-1}C_0^2\mu^2\kappa^2r^3/(n_1 \wedge n_2)$, then our initialization in Algorithm \ref{alg:pgd} satisfies
\begin{eqnarray*}
\dist(\pmb{Z}^1,\pmb{Z}^*) \leq \dist(\pmb{Z}^0,\pmb{Z}^*) \leq \frac{1}{4}\sqrt{\sigma_r^*}.
\end{eqnarray*}
\end{theorem}

\begin{corollary}\label{r_1}
The ground-truth $\pmb{Z}^*$ satisfies $\Vert \pmb{Z}^*\Vert_{2,\infty}\leq\sqrt{\frac{2\mu r}{n_1 \wedge n_2}}\Vert \pmb{Z}^0\Vert$. 
\end{corollary}

\begin{proof}[Proof of Theorem \ref{1}]
Based on Lemma \ref{l_1} and Lemma \ref{l_2}, we have 
\begin{align}
\Vert \pmb{Z}^0(\pmb{Z}^0)^\tp - \pmb{Z}^*(\pmb{Z}^*)^\tp\Vert_F 
 \leq& 2 \Vert \pmb{U}^0\pmb{\Sigma}^0 (\pmb{V}^0)^\tp - \pmb{U}^*\pmb{\Sigma}^* (\pmb{V}^*)^\tp\Vert_F
\overset{(i)}{\leq} 2\sqrt{2r}\Vert \pmb{U}^0\pmb{\Sigma}^0 (\pmb{V}^0)^\tp - \pmb{M}^*\Vert\nonumber\\
\leq& 2\sqrt{2r}\cdot\frac{2C_0\mu r}{\sqrt{d}}\Vert \pmb{M}^*\Vert
= \frac{4\sqrt{2}C_0\mu r^{1.5}}{\sqrt{d}}\Vert \pmb{M}^*\Vert. \label{3p1}
\end{align}
Here, the inequality (i) holds because of Cauchy-Schwarz inequality and $\text{rank}(\pmb{U}^0\pmb{\Sigma}^0 (\pmb{V}^0)^\tp - \pmb{U}^*\pmb{\Sigma}^* (\pmb{V}^*)^\tp) \leq \text{rank}(\pmb{U}^0\pmb{\Sigma}^0 (\pmb{V}^0)^\tp) + \text{rank}(\pmb{U}^*\pmb{\Sigma}^* (\pmb{V}^*)^\tp) = 2r$.
On the other hand, we set $\pmb{H}^0=\pmb{Z}^0-\overline{\pmb{Z}}{}^0$ and from Lemma 5.4 in \cite{TBS2016}, we have
\begin{eqnarray}
\Vert \pmb{Z}^0(\pmb{Z}^0)^\tp - \pmb{Z}^*(\pmb{Z}^*)^\tp\Vert_F^2 \geq 4(\sqrt{2}-1)\sigma_r^*\Vert \pmb{H}^0\Vert_F^2.\label{3p2}
\end{eqnarray}
Combining with (\ref{3p1}) and (\ref{3p2}), we can get
\begin{align}
&\dist(\pmb{Z}^0,\pmb{Z}^*)^2 = \Vert \pmb{H}^0 \Vert_F^2 \nonumber\\
 \leq& \frac{\Vert \pmb{Z}^0(\pmb{Z}^0)^\tp - \pmb{Z}^*(\pmb{Z}^*)^\tp\Vert_F^2}{4(\sqrt{2}-1)\sigma_r^*}
\leq\frac{8}{\sqrt{2}-1}\cdot\frac{C_0^2\mu^2r^3}{d\sigma_r^*}\sigma_1^*{}^2\nonumber\\
=&\frac{8}{\sqrt{2}-1}\cdot\frac{C_0^2\mu^2r^3}{d}\cdot\kappa^2\cdot\sigma_r^*\nonumber\\
\leq&\frac{1}{16}\sigma_r^*, \label{3p3}
\end{align}
where the last inequality comes from $d \geq \frac{128}{\sqrt{2}-1}C_0^2\mu^2\kappa^2r^3$.
Since $\mathcal{P}_{\mathcal{C}_1}$ is row-wise clipping, we apply Lemma \ref{l_3} to get 
\begin{eqnarray*}
\dist(\pmb{Z}^1,\pmb{Z}^*)^2 \leq \Vert \mathcal{P}_{\mathcal{C}_1}(\pmb{Z}^0)-\overline{\pmb{Z}}{}^0\Vert_F^2 
\leq \Vert \pmb{Z}^0-\overline{\pmb{Z}}{}^0\Vert_F^2 
= \dist(\pmb{Z}^0,\pmb{Z}^*)^2 \leq \frac{1}{16}\sigma_r^*.
\end{eqnarray*}
\end{proof}

\begin{proof}[Proof of Corollary \ref{r_1}]
From the incoherence condition ${\rm \pmb{\mathscr{A}}_1}$, we have
\begin{eqnarray*}
\Vert \pmb{Z}^*\Vert_{2,\infty}\leq\sqrt{\frac{\mu r}{n}}\cdot\sqrt{\sigma_1^*}.
\end{eqnarray*}
Thus, we need only to prove $\sigma_1^* \leq 2\sigma_1(\pZ^0)$. From (\ref{3p3}) in the proof of Theorem \ref{1}, it goes that
\begin{eqnarray*}
\frac{8}{\sqrt{2}-1}\cdot\frac{C_0^2\mu^2r^3}{d\sigma_r^*}\sigma_1^*{}^2 \leq \frac{1}{16}\sigma_r^*.
\end{eqnarray*}
After transposition, we get
\begin{align*}
\left\Vert \frac{1}{p}\mathcal{P}_{\Omega}(\pmb{M}^*)-\pmb{M}^*\right\Vert^2 \leq \frac{C_0^2\mu^2r^2}{d}\sigma_1^*{}^2 
 \leq \frac{\sqrt{2}-1}{8r}\cdot\frac{1}{16}\sigma_r^*{}^2\leq \left( \frac{1}{16}\sigma_r^*\right)^2
\end{align*}
because of Lemma \ref{l_1}. Then, by Weyl' s theorem, we have
\begin{eqnarray}
\vert \sigma_1(\pZ^0)-\sigma_1^*\vert\leq\frac{1}{16}\sigma_r^*. \label{weyl}
\end{eqnarray}
That means $\sigma_1^* \leq 2\sigma_1(\pZ^0)$, which completes the proof of Corollary \ref{r_1}.
\end{proof}

\subsection*{A.2 Local Geometric Condition} \label{subsec:local1}
Before giving the proof of main convergence theorem,
we need to state that the loss function $\mathcal{L}_1(\pmb{Z})$ satisfies local geometric condition \cite{CLS2015} which indicates the gradient of $\mathcal{L}_1(\pmb{Z})$ is well-behaved. Similar definition can also be found in \cite{T2}. 

\begin{itemize}
\item local curvature condition:
\begin{proposition} \label{p_1} 
If $p \geq 262144C_0^2\mu^2r^2\kappa^2/(n_1 \wedge n_2)$, 
by setting $\delta_d =\frac{1}{64}$, $ \lambda = \frac{1}{2}$,  then for any $\pmb{Z} \in \mathcal{C}_1$ satisfying $\dist(\pmb{Z},\pmb{Z}^*) \leq \frac{1}{4}\sqrt{\sigma_r^*}$, we have 
\begin{eqnarray}
\langle \nabla  \mL_1(\pmb{Z}), \pmb{H} \rangle \geq \frac{3}{8}\sigma^*_r\Vert\pmb{H} \Vert_F^2+\frac{1}{4} \Vert (\overline{\pmb{Z}})^\tp \pmb{D}\pmb{H}\Vert_F^2. \label{lcc}
\end{eqnarray}
\end{proposition}

\item local smoothness condition: 
\begin{proposition}\label{p_2}
If $p \geq 262144C_0^2\mu^2r^2\kappa^2/(n_1 \wedge n_2)$, 
by setting $\delta_d =\frac{1}{64}$, $ \lambda = \frac{1}{2}$, then for any $\pmb{Z} \in \mathcal{C}_1$ satisfying $\dist(\pmb{Z},\pmb{Z}^*) \leq \frac{1}{4}\sqrt{\sigma_r^*}$, we have
\begin{eqnarray}
\Vert \nabla  \mL_1(\pmb{Z}) \Vert_F^2 \leq 1524\mu^2r^2\sigma_1^*{}^2\Vert \pmb{H}\Vert_F^2 + 5\sigma_1^*\Vert (\overline{\pmb{Z}})^\tp \pmb{D}\pmb{H}\Vert_F^2. \label{lsc}
\end{eqnarray}
\end{proposition}

\end{itemize}

Rearranging the terms in the local smoothness condition (\ref{lsc}), we can obtain
\begin{align}
\Vert (\overline{\pmb{Z}})^\tp \pmb{D}\pmb{H}\Vert_F^2 
\geq& \frac{\Vert \nabla  \mL_1(\pmb{Z})\Vert_F^2}{5\sigma_1^*} -\frac{1524}{5}\mu^2r^2\sigma_1^*\Vert \pmb{H}\Vert_F^2\nonumber \\
 \geq& \frac{\Vert \nabla  \mL_1(\pmb{Z})\Vert_F^2}{1524\mu^2 r^2 \kappa \sigma_1^*} - \sigma_r^*\Vert \pmb{H}\Vert_F^2. \label{lsc4}
\end{align}
Combining (\ref{lcc}) and (\ref{lsc4}), we can obtain that $ \mL_1(\pmb{Z})$ satisfies the regularity condition:
\begin{eqnarray}
\langle \nabla  \mL_1(\pmb{Z}), \pmb{H} \rangle \geq \frac{1}{8}\sigma^*_r\Vert\pmb{H} \Vert_F^2+ \frac{\Vert \nabla  \mL_1(\pmb{Z})\Vert_F^2}{6096\mu^2 r^2 \kappa \sigma_1^*}. \label{lsc5}
\end{eqnarray}

\begin{proof}[Proof of Proposition \ref{p_1}]
From \eqref{lo5}, we can compute
\begin{align*}
\langle \nabla  \mL_1(\pmb{Z}), \pmb{H} \rangle  
=& \frac{1}{p}\sum_{l=1}^{2m}\left(\langle \pmb{A}_l,\pmb{H}(\overline{\pmb{Z}})^\tp+\overline{\pmb{Z}}\pmb{H}^\tp+\pmb{H}\pmb{H}^\tp\rangle \right. \\
& \left. \cdot \langle(\pmb{A}_l+\pmb{A}_l^\tp)(\overline{\pmb{Z}}+\pmb{H}), \pmb{H}\rangle\right)
+\lambda \text{tr}(\pmb{H}^\tp\pmb{D}\pmb{Z}\pmb{Z}^\tp \pmb{D}\pmb{Z})\\
=&\frac{1}{p}\sum_{l=1}^{2m}\langle \pmb{A}_l,\pmb{H}(\overline{\pmb{Z}})^\tp+\overline{\pmb{Z}}\pmb{H}^\tp\rangle^2 + \frac{2}{p}\sum_{l=1}^{2m}\langle \pmb{A}_l, \pmb{H}\pmb{H}^\tp\rangle^2\\
&+\frac{3}{p}\sum_{l=1}^{2m}\langle \pmb{A}_l,\pmb{H}(\overline{\pmb{Z}})^\tp+\overline{\pmb{Z}}\pmb{H}^\tp\rangle\langle\pmb{A}_l, \pmb{H}\pmb{H}^\tp\rangle 
 +\lambda \text{tr}(\pmb{H}^\tp\pmb{D}\pmb{Z}\pmb{Z}^\tp \pmb{D}\pmb{Z}).
\end{align*}
Cauchy-Schwarz inequality yields that
\begin{align*}
 \langle \nabla  \mL_1(\pmb{Z}), \pmb{H} \rangle
\geq& \frac{1}{p}\sum_{l=1}^{2m}\langle \pmb{A}_l,\pmb{H}(\overline{\pmb{Z}})^\tp+\overline{\pmb{Z}}\pmb{H}^\tp\rangle^2 + \frac{1}{p}\sum_{l=1}^{2m}2\langle \pmb{A}_l, \pmb{H}\pmb{H}^\tp\rangle^2\\
&-\frac{3}{\sqrt{2}}\frac{1}{p}\sqrt{\sum_{l=1}^{2m}\langle \pmb{A}_l,\pmb{H}(\overline{\pmb{Z}})^\tp+\overline{\pmb{Z}}\pmb{H}^\tp\rangle^2}\sqrt{\sum_{l=1}^{2m}2\langle \pmb{A}_l, \pmb{H}\pmb{H}^\tp\rangle^2} \\
&+ \lambda \text{tr}(\pmb{H}^\tp\pmb{D}\pmb{Z}\pmb{Z}^\tp \pmb{D}\pmb{Z}).
\end{align*}
We then utilize the fundamental inequality $(a-b)^2\geq\frac{a^2}{2}-b^2$ to get
\begin{align*}
 \langle \nabla  \mL_1(\pmb{Z}), \pmb{H} \rangle 
\geq& \frac{1}{2p}\sum_{l=1}^{2m}\langle \pmb{A}_l,\pmb{H}(\overline{\pmb{Z}})^\tp+\overline{\pmb{Z}}\pmb{H}^\tp\rangle^2 - \frac{5}{2p}\sum_{l=1}^{2m}\langle \pmb{A}_l, \pmb{H}\pmb{H}^\tp\rangle^2
+ \lambda \text{tr}(\pmb{H}^\tp\pmb{D}\pmb{Z}\pmb{Z}^\tp \pmb{D}\pmb{Z})\\
\geq&\frac{1}{2p}\Vert \mathcal{P}_{\underline{\Omega}}(\pmb{H}(\overline{\pmb{Z}})^\tp+\overline{\pmb{Z}}\pmb{H}^\tp)\Vert_F^2 - \frac{5}{2p}\Vert \mathcal{P}_{\underline{\Omega}}(\pmb{H}\pmb{H}^\tp)\Vert_F^2
+ \lambda \text{tr}(\pmb{H}^\tp\pmb{D}\pmb{Z}\pmb{Z}^\tp \pmb{D}\pmb{Z}).
\end{align*}

Firstly, we turn to estimate $\frac{1}{2p}\Vert \mathcal{P}_{\underline{\Omega}}(\pmb{H}(\overline{\pmb{Z}})^\tp+\overline{\pmb{Z}}\pmb{H}^\tp)\Vert_F^2$. By the symmetry of $\underline{\Omega}$, we have
\begin{align*}
\frac{1}{2p}\Vert \mathcal{P}_{\underline{\Omega}}(\pmb{H}(\overline{\pmb{Z}})^\tp+\overline{\pmb{Z}}\pmb{H}^\tp)\Vert_F^2
 =& \frac{1}{p}\Vert \mathcal{P}_\Omega (\pmb{H}_{\pmb{U}}(\overline{\pmb{Z}}_{\pmb{V}})^\tp + \overline{\pmb{Z}}_{\pmb{U}}(\pmb{H}_{\pmb{V}})^\tp)\Vert_F^2\\
=&\frac{1}{p}\Vert \mathcal{P}_\Omega (\pmb{H}_{\pmb{U}}(\overline{\pmb{Z}}_{\pmb{V}})^\tp)\Vert_F^2 + \frac{1}{p}\Vert \mathcal{P}_\Omega (\overline{\pmb{Z}}_{\pmb{U}}(\pmb{H}_{\pmb{V}})^\tp)\Vert_F^2 \\
&+ \frac{2}{p}\langle \mathcal{P}_\Omega (\pmb{H}_{\pmb{U}}(\overline{\pmb{Z}}_{\pmb{V}})^\tp), \mathcal{P}_\Omega (\overline{\pmb{Z}}_{\pmb{U}}(\pmb{H}_{\pmb{V}})^\tp)\rangle.
\end{align*}
By the settings of $\pmb{H}$ and $\pmb{Z}$, we can clarify that $\pmb{H}_{\pmb{U}}(\overline{\pmb{Z}}_{\pmb{V}})^\tp$ and $\overline{\pmb{Z}}_{\pmb{U}}(\pmb{H}_{\pmb{V}})^\tp$ belong to the space $T$ in \eqref{def_T}. Then, based on Lemma \ref{l_5}, we have
\begin{align}
\frac{1}{2p}\Vert \mathcal{P}_{\underline{\Omega}}(\pmb{H}(\overline{\pmb{Z}})^\tp+\overline{\pmb{Z}}\pmb{H}^\tp)\Vert_F^2 
\geq& (1-\tilde{\delta})\left( \Vert \pmb{H}_{\pmb{U}}(\overline{\pmb{Z}}_{\pmb{V}})^\tp\Vert_F^2+\Vert \overline{\pmb{Z}}_{\pmb{U}}(\pmb{H}_{\pmb{V}})^\tp\Vert_F^2 \right)\nonumber\\
&  -2\tilde{\delta} \Vert \pmb{H}_{\pmb{U}}(\overline{\pmb{Z}}_{\pmb{V}})^\tp \Vert_F \Vert \overline{\pmb{Z}}_{\pmb{U}}(\pmb{H}_{\pmb{V}})^\tp\Vert_F 
 + 2\langle \pmb{H}_{\pmb{U}}(\overline{\pmb{Z}}_{\pmb{V}})^\tp,  \overline{\pmb{Z}}_{\pmb{U}}(\pmb{H}_{\pmb{V}})^\tp\rangle  \nonumber\\
\overset{(i)}{\geq}&  (1-2\tilde{\delta})\left( \Vert \pmb{H}_{\pmb{U}}(\overline{\pmb{Z}}_{\pmb{V}})^\tp\Vert_F^2+\Vert \overline{\pmb{Z}}_{\pmb{U}}(\pmb{H}_{\pmb{V}})^\tp\Vert_F^2 \right) \nonumber\\
&+ 2\langle \pmb{H}_{\pmb{U}}(\overline{\pmb{Z}}_{\pmb{V}})^\tp,  \overline{\pmb{Z}}_{\pmb{U}}(\pmb{H}_{\pmb{V}})^\tp\rangle, \label{lcc1}
\end{align}
where $\tilde{\delta} = \sqrt{2(\delta_d^2+\frac{C_0^2\mu^2r^2}{d})}$ and  the inequality $(i)$ holds because $2\Vert \pmb{H}_{\pmb{U}}(\overline{\pmb{Z}}_{\pmb{V}})^\tp \Vert_F \Vert \overline{\pmb{Z}}_{\pmb{U}}(\pmb{H}_{\pmb{V}})^\tp\Vert_F \leq \Vert \pmb{H}_{\pmb{U}}(\overline{\pmb{Z}}_{\pmb{V}})^\tp\Vert_F^2+\Vert \overline{\pmb{Z}}_{\pmb{U}}(\pmb{H}_{\pmb{V}})^\tp\Vert_F^2$.\\
Recalling that $\overline{\pmb{Z}}_{\pmb{V}} = \pmb{V}^*(\pmb{\Sigma}^*)^{\frac{1}{2}}\pmb{R}$ and $ \overline{\pmb{Z}}_{\pmb{U}} = \pmb{U}^*(\pmb{\Sigma}^*)^{\frac{1}{2}}\pmb{R}$ where $\pmb{R}$ is a orthogonal matrix, we have
$\sigma (\overline{\pmb{Z}}_{\pmb{V}})=\sqrt{\sigma^*_r}, \sigma (\overline{\pmb{Z}}_{\pmb{U}})=\sqrt{\sigma^*_r}$.
Thus, it goes that
\begin{eqnarray}
 \Vert \pmb{H}_{\pmb{U}}(\overline{\pmb{Z}}_{\pmb{V}})^\tp\Vert_F^2 \geq \sigma^*_r \Vert \pmb{H}_{\pmb{U}}\Vert_F^2, \Vert \overline{\pmb{Z}}_{\pmb{U}}(\pmb{H}_{\pmb{V}})^\tp\Vert_F^2 \geq \sigma^*_r \Vert \pmb{H}_{\pmb{V}}\Vert_F^2.\label{lcc2}
\end{eqnarray}
Combining with (\ref{lcc1}) and (\ref{lcc2}), we have 
\begin{align*}
\frac{1}{2p}\Vert \mathcal{P}_{\underline{\Omega}}(\pmb{H}(\overline{\pmb{Z}})^\tp+\overline{\pmb{Z}}\pmb{H}^\tp)\Vert_F^2 
\geq& (1-2\tilde{\delta})\sigma^*_r\Vert\pmb{H} \Vert_F^2 + 2\langle \pmb{H}_{\pmb{U}}(\overline{\pmb{Z}}_{\pmb{V}})^\tp,  \overline{\pmb{Z}}_{\pmb{U}}(\pmb{H}_{\pmb{V}})^\tp\rangle.
\end{align*}
Thus, it goes that
\begin{align}
\langle \nabla  \mL_1(\pmb{Z}), \pmb{H} \rangle 
\geq&   \underbrace{ 2\langle \pmb{H}_{\pmb{U}}(\overline{\pmb{Z}}_{\pmb{V}})^\tp,  \overline{\pmb{Z}}_{\pmb{U}}(\pmb{H}_{\pmb{V}})^\tp\rangle +  \lambda \text{tr}(\pmb{H}^\tp\pmb{D}\pmb{Z}\pmb{Z}^\tp \pmb{D}\pmb{Z})}_{A_1}  \nonumber\\
 & - \frac{5}{2p}\Vert \mathcal{P}_{\underline{\Omega}}(\pmb{H}\pmb{H}^\tp)\Vert_F^2 + (1-2\tilde{\delta})\sigma^*_r\Vert\pmb{H} \Vert_F^2. \label{lcc4}
\end{align}

Secondly, we turn to bound $A_1$.
Since
\begin{align*}
2\langle \pmb{H}_{\pmb{U}}(\overline{\pmb{Z}}_{\pmb{V}})^\tp,  \overline{\pmb{Z}}_{\pmb{U}}(\pmb{H}_{\pmb{V}})^\tp\rangle =& \left\langle \left[                 
  \begin{array}{ccc}   
   \pmb{H}_{\pmb{U}}\\  
   \pmb{H}_{\pmb{V}}\\  
  \end{array}
\right], \left[                 
  \begin{array}{ccc}   
    \overline{\pmb{Z}}_{\pmb{U}}(\pmb{H}_{\pmb{V}})^\tp\overline{\pmb{Z}}_{\pmb{V}}\\  
     \overline{\pmb{Z}}_{\pmb{V}}(\pmb{H}_{\pmb{U}})^\tp\overline{\pmb{Z}}_{\pmb{U}}\\  
  \end{array}
\right] \right\rangle  \\
=& \left \langle \pmb{H}, \left[                 
  \begin{array}{ccc}   
    \pmb{0} &  \overline{\pmb{Z}}_{\pmb{U}}(\pmb{H}_{\pmb{V}})^\tp\\  
     \overline{\pmb{Z}}_{\pmb{V}}(\pmb{H}_{\pmb{U}})^\tp & \pmb{0}\\  
  \end{array}
\right]\overline{\pmb{Z}}\right \rangle \\
=& \left\langle \pmb{H}, \frac{1}{2}(\overline{\pmb{Z}}\pmb{H}^\tp-\pmb{D}\overline{\pmb{Z}}\pmb{H}^\tp \pmb{D})\overline{\pmb{Z}}\right \rangle,
\end{align*}
we can get
\bea
A_1 = \langle \pmb{H}, \frac{1}{2}(\overline{\pmb{Z}}\pmb{H}^\tp-\pmb{D}\overline{\pmb{Z}}\pmb{H}^\tp \pmb{D})\overline{\pmb{Z}} \rangle + \lambda\langle \pmb{H}, \pmb{D}\pmb{Z}\pmb{Z}^\tp \pmb{D}\pmb{Z} \rangle.
\eea
Since $\pmb{Z} = \overline{\pmb{Z}}+\pmb{H}$, $(\overline{\pmb{Z}})^\tp \pmb{D}\overline{\pmb{Z}} = \pmb{0}$ and $\pmb{Z}\pmb{Z}^\tp-\overline{\pmb{Z}}(\overline{\pmb{Z}})^\tp = \pmb{H}\pmb{H}^\tp + \overline{\pmb{Z}}\pmb{H}^\tp+\pmb{H}(\overline{\pmb{Z}})^\tp$, we have
\begin{align*}
A_1 =& \langle \pmb{H}, \frac{1}{2}(\overline{\pmb{Z}}\pmb{H}^\tp-\pmb{D}\overline{\pmb{Z}}\pmb{H}^\tp \pmb{D})\overline{\pmb{Z}} \rangle+ \lambda \Vert (\overline{\pmb{Z}})^\tp \pmb{D}\pmb{H}\Vert_F^2\\
& + \lambda\langle \pmb{H}, \pmb{D}(\pmb{H}\pmb{H}^\tp + \overline{\pmb{Z}}\pmb{H}^\tp+\pmb{H}(\overline{\pmb{Z}})^\tp) \pmb{D}(\overline{\pmb{Z}}+\pmb{H}) \rangle.
\end{align*}
After simple computations, it arrives at
\begin{align*}
A_1
=& \frac{\lambda}{2} \Vert (\overline{\pmb{Z}})^\tp \pmb{D}\pmb{H}\Vert_F^2 + \frac{1}{2}\Vert \pmb{H}^\tp\overline{\pmb{Z}}\Vert_F^2 - \frac{7}{2}\lambda\Vert \pmb{H}^\tp \pmb{D}\pmb{H}\Vert_F^2 \\
&+ \frac{\lambda}{2} \Vert (\overline{\pmb{Z}})^\tp \pmb{D}\pmb{H}+3\pmb{H}^\tp \pmb{D}\pmb{H}\Vert_F^2 
+ (\lambda -\frac{1}{2})\text{tr} (\pmb{H}^\tp \pmb{D}\overline{\pmb{Z}}\pmb{H}^\tp \pmb{D}\overline{\pmb{Z}})\\
\geq&  \frac{\lambda}{2} \Vert (\overline{\pmb{Z}})^\tp \pmb{D}\pmb{H}\Vert_F^2  - \frac{7}{2}\lambda\Vert \pmb{H}^\tp \pmb{D}\pmb{H}\Vert_F^2
 + (\lambda -\frac{1}{2})\text{tr} (\pmb{H}^\tp \pmb{D}\overline{\pmb{Z}}\pmb{H}^\tp \pmb{D}\overline{\pmb{Z}}).
\end{align*}
 By setting $\lambda = \frac{1}{2}$, we can get
 \begin{eqnarray}
 A_1 \geq \frac{1}{4} \Vert (\overline{\pmb{Z}})^\tp \pmb{D}\pmb{H}\Vert_F^2  - \frac{7}{4}\Vert \pmb{H}\Vert_F^4. \label{lcc5}
 \end{eqnarray}

At last, we turn to bound $\frac{1}{p}\Vert \mathcal{P}_{\underline{\Omega}}(\pmb{H}\pmb{H}^\tp)\Vert_F^2$.
Based on Corollary \ref{r_1}, equation \eqref{weyl} and $\pZ \in \mathcal{C}_1$, we can get
\be
\| \pmb{H} \|_{2,\infty} \leq \| \pmb{Z} \|_{2,\infty} + \| \overline{\pmb{Z}} \|_{2,\infty} \leq 4\sqrt{\frac{\mu r \sigma_1^*}{n}}. \label{H_2infty}
\ee
Then using Lemma \ref{l_7} by choosing $p \geq 262144C_0^2\mu^2r^2\kappa^2/n$, we have 
   \begin{eqnarray}
  \frac{1}{p}\Vert \mathcal{P}_{\underline{\Omega}}(\pmb{H}\pmb{H}^\tp)\Vert_F^2 \leq 2\Vert \pmb{H}\Vert_F^4 + \frac{1}{16}\sigma_r^*\Vert \pmb{H}\Vert_F^2.  \label{lcc6}
   \end{eqnarray}

 Combining (\ref{lcc4}), (\ref{lcc5}), (\ref{lcc6}) with $\delta_d \leq \frac{1}{64}$ and $p \geq 262144C_0^2\mu^2r^2\kappa^2/n$, we can get
 \begin{align*}
 \langle \nabla  \mL_1(\pmb{Z}), \pmb{H} \rangle
 \geq& \left(1-2\sqrt{2\left(\delta_d^2+\frac{C_0^2\mu^2r^2}{d}\right)}\right)\sigma^*_r\Vert\pmb{H} \Vert_F^2 +\frac{1}{4} \Vert (\overline{\pmb{Z}})^\tp \pmb{D}\pmb{H}\Vert_F^2  \nonumber\\  
 &- \frac{7}{4}\Vert \pmb{H}\Vert_F^4  - \frac{5}{2}\left(2\Vert \pmb{H}\Vert_F^4 + \frac{1}{16}\sigma_r^*\Vert \pmb{H}\Vert_F^2\right)\\
 \geq& \left(1-\frac{3}{64}- \frac{5}{32}\right)\sigma^*_r\Vert\pmb{H} \Vert_F^2 - \frac{27}{4}\Vert \pmb{H}\Vert_F^4+\frac{1}{4} \Vert (\overline{\pmb{Z}})^\tp \pmb{D}\pmb{H}\Vert_F^2.
 \end{align*}
Since $\dist(\pZ, \pZ^*) \leq \frac{1}{4}\sqrt{\sigma_r^*}$, 
the local curvature condition follows that
\begin{eqnarray}
 \langle \nabla  \mL_1(\pmb{Z}), \pmb{H} \rangle \geq \frac{3}{8}\sigma^*_r\Vert\pmb{H} \Vert_F^2+\frac{1}{4} \Vert (\overline{\pmb{Z}})^\tp \pmb{D}\pmb{H}\Vert_F^2. \label{lcc7}
\end{eqnarray}
\end{proof}

\begin{proof}[Proof of Proposition \ref{p_2}]
Since the Frobenius norm satisfies 
\begin{eqnarray*}
\Vert \nabla  \mL_1(\pmb{Z})\Vert_F^2 = \max_{\Vert \pmb{W}\Vert_F=1}\vert \langle\nabla  \mL_1(\pmb{Z}), \pmb{W}\rangle\vert^2,
\end{eqnarray*}
thus we focus on the inner product $\langle\nabla  \mL_1(\pmb{Z}), \pmb{W}\rangle$ when $\Vert \pmb{W}\Vert_F=1$,
\begin{align*}
\langle\nabla  \mL_1(\pmb{Z}), \pmb{W}\rangle 
=&  \frac{1}{p}\sum_{l=1}^{2m}\left(\langle \pmb{A}_l,\pmb{H}(\overline{\pmb{Z}})^\tp+\overline{\pmb{Z}}\pmb{H}^\tp+\pmb{H}\pmb{H}^\tp\rangle 
 \cdot\langle(\pmb{A}_l+\pmb{A}_l^\tp)(\overline{\pmb{Z}}+\pmb{H}), \pmb{W}\rangle\right)
\\&+ \lambda\langle \pmb{D}\pmb{Z}\pmb{Z}^\tp \pmb{D}\pmb{Z}, \pmb{W}\rangle\\
=&\frac{1}{p} \langle \mathcal{P}_{\underline{\Omega}}(\pmb{H}(\overline{\pmb{Z}})^\tp+\overline{\pmb{Z}}\pmb{H}^\tp), \mathcal{P}_{\underline{\Omega}}( \pmb{W}(\overline{\pmb{Z}})^\tp+\overline{\pmb{Z}}\pmb{W}^\tp)\rangle \\
& + \frac{1}{p} \langle \mathcal{P}_{\underline{\Omega}}(\pmb{H}\pmb{H}^\tp), \mathcal{P}_{\underline{\Omega}}( \pmb{W}(\overline{\pmb{Z}})^\tp+\overline{\pmb{Z}}\pmb{W}^\tp)\rangle \\
& +\frac{1}{p} \langle \mathcal{P}_{\underline{\Omega}}(\pmb{H}(\overline{\pmb{Z}})^\tp+\overline{\pmb{Z}}\pmb{H}^\tp), \mathcal{P}_{\underline{\Omega}}(\pmb{W}\pmb{H}^\tp+\pmb{H}\pmb{W}^\tp)\rangle \\
&+ \frac{1}{p} \langle \mathcal{P}_{\underline{\Omega}}(\pmb{H}\pmb{H}^\tp), \mathcal{P}_{\underline{\Omega}}(\pmb{W}\pmb{H}^\tp+\pmb{H}\pmb{W}^\tp)\rangle 
+ \lambda\langle \pmb{D}\pmb{Z}\pmb{Z}^\tp \pmb{D}\pmb{Z}, \pmb{W}\rangle.
\end{align*}
Since $(a+b+c+d+e)^2\leq 5(a^2+b^2+c^2+d^2+e^2)$, we have
\begin{align*}
\vert \langle\nabla  \mL_1(\pmb{Z}), \pmb{W}\rangle\vert^2
\leq&\frac{5}{p^2} \langle \mathcal{P}_{\underline{\Omega}}(\pmb{H}(\overline{\pmb{Z}})^\tp+\overline{\pmb{Z}}\pmb{H}^\tp), \mathcal{P}_{\underline{\Omega}}( \pmb{W}(\overline{\pmb{Z}})^\tp+\overline{\pmb{Z}}\pmb{W}^\tp)\rangle^2\\
& + \frac{5}{p^2} \langle \mathcal{P}_{\underline{\Omega}}(\pmb{H}\pmb{H}^\tp), \mathcal{P}_{\underline{\Omega}}( \pmb{W}(\overline{\pmb{Z}})^\tp+\overline{\pmb{Z}}\pmb{W}^\tp)\rangle^2\\
&+ \frac{5}{p^2}\langle \mathcal{P}_{\underline{\Omega}}(\pmb{H}(\overline{\pmb{Z}})^\tp+\overline{\pmb{Z}}\pmb{H}^\tp), \mathcal{P}_{\underline{\Omega}}(\pmb{W}\pmb{H}^\tp+\pmb{H}\pmb{W}^\tp)\rangle^2\\
&+ \frac{5}{p^2} \langle \mathcal{P}_{\underline{\Omega}}(\pmb{H}\pmb{H}^\tp), \mathcal{P}_{\underline{\Omega}}(\pmb{W}\pmb{H}^\tp+\pmb{H}\pmb{W}^\tp)\rangle^2  
+ 5\lambda^2\langle \pmb{D}\pmb{Z}\pmb{Z}^\tp \pmb{D}\pmb{Z}, \pmb{W}\rangle^2.
\end{align*}
Cauchy-Schwarz inequality yields that
\begin{align*}
\vert \langle\nabla  \mL_1(\pmb{Z}), \pmb{W}\rangle\vert^2
\leq& \frac{5}{p^2}\left[ \Vert \mathcal{P}_{\underline{\Omega}}(\pmb{H}(\overline{\pmb{Z}})^\tp+\overline{\pmb{Z}}\pmb{H}^\tp)\Vert_F^2 +  \Vert \mathcal{P}_{\underline{\Omega}}(\pmb{H}\pmb{H}^\tp)\Vert_F^2 \right] \\
&\cdot \left[ \Vert \mathcal{P}_{\underline{\Omega}}( \pmb{W}(\overline{\pmb{Z}})^\tp+\overline{\pmb{Z}}\pmb{W}^\tp)\Vert_F^2 +  \Vert \mathcal{P}_{\underline{\Omega}}(\pmb{W}\pmb{H}^\tp+\pmb{H}\pmb{W}^\tp)\Vert_F^2 \right] \\
&+ 5\lambda^2\Vert \pmb{W}\Vert_F^2\Vert \pmb{D}\pmb{Z}\pmb{Z}^\tp \pmb{D}\pmb{Z}\Vert_F^2\\
\overset{(i)}{\leq}& \frac{5}{p}\left[ 4\Vert \mathcal{P}_{\underline{\Omega}}(\pmb{H}(\overline{\pmb{Z}})^\tp)\Vert_F^2 + \Vert \mathcal{P}_{\underline{\Omega}}(\pmb{H}\pmb{H}^\tp)\Vert_F^2\right] \\
&\cdot  \frac{1}{p} \left[ 4\Vert \mathcal{P}_{\underline{\Omega}}(\pmb{W}(\overline{\pmb{Z}})^\tp)\Vert_F^2+4\Vert \mathcal{P}_{\underline{\Omega}}(\pmb{W}\pmb{H}^\tp)\Vert_F^2\right]
+5\lambda^2 \Vert \pmb{D}\pmb{Z}\pmb{Z}^\tp \pmb{D}\pmb{Z}\Vert_F^2,
\end{align*}
where  step $(i)$ follows from $(a+b)^2 \leq 2(a^2+b^2)$ and $\Vert \pmb{W}\Vert_F^2 = 1$.
Based on Lemmas \ref{l_8} and \ref{l_7} with $p \geq 262144C_0^2\mu^2r^2\kappa^2/n$, 
\begin{align}
\vert \langle\nabla  \mL_1(\pmb{Z}), \pmb{W}\rangle\vert^2
\leq& 5\left( 4n\Vert \pmb{H}\Vert_F^2\Vert \overline{\pmb{Z}}\Vert_{2,\infty}^2 + 2\Vert \pmb{H}\Vert_F^4 + \frac{1}{16}\sigma_r^*\Vert \pmb{H}\Vert_F^2\right) \nonumber\\
&\cdot \left( 4n\Vert \pmb{W}\Vert_F^2\Vert \overline{\pmb{Z}}\Vert_{2,\infty}^2+4n\Vert \pmb{W}\Vert_F^2\Vert \pmb{H}\Vert_{2,\infty}^2\right) 
+ 5\lambda^2 \Vert \pmb{D}\pmb{Z}\pmb{Z}^\tp \pmb{D}\pmb{Z}\Vert_F^2\nonumber\\
\overset{(ii)}{\leq}& \left( 20\mu r\sigma_1^*+ \frac{10}{16}\sigma_r^* + \frac{5}{16}\sigma_r^*\right)\cdot \Vert \pmb{H}\Vert_F^2 \cdot \left( 68\mu r\sigma_1^*\right) 
+ 5\lambda^2 \Vert \pmb{D}\pmb{Z}\pmb{Z}^\tp \pmb{D}\pmb{Z}\Vert_F^2,\label{lsc1}
\end{align}
where the inequality $(ii)$ holds because of incoherence condition, $\Vert \pmb{H}\Vert_F \leq \frac{1}{4}\sqrt{\sigma_r^*}$ and $\Vert \pmb{H}\Vert_{2,\infty}\leq 4\sqrt{\frac{\mu r}{n}\sigma_1^*}$ in \eqref{H_2infty}.

As for $\Vert \pmb{D}\pmb{Z}\pmb{Z}^\tp \pmb{D}\pmb{Z}\Vert_F^2$, we utilize the basic inequality to obtain
\begin{align*}
\Vert \pmb{D}\pmb{Z}\pmb{Z}^\tp \pmb{D}\pmb{Z}\Vert_F^2
\leq& 2\Vert \pmb{D}(\pmb{Z}\pmb{Z}^\tp -\overline{\pmb{Z}}(\overline{\pmb{Z}})^\tp)\pmb{D}\pmb{Z}\Vert_F^2 + 2\Vert \pmb{D}\overline{\pmb{Z}} (\overline{\pmb{Z}})^\tp \pmb{D}\pmb{Z}\Vert_F^2\\
\leq& 2\Vert \pmb{Z}\Vert^2\cdot\Vert \pmb{Z}\pmb{Z}^\tp -\overline{\pmb{Z}}(\overline{\pmb{Z}})^\tp\Vert_F^2 + 2\Vert \overline{\pmb{Z}}\Vert^2\cdot \Vert (\overline{\pmb{Z}})^\tp \pmb{D}\pmb{Z}\Vert_F^2\\
=&  2\Vert \pmb{Z}\Vert^2\cdot\Vert \pmb{H}(\overline{\pmb{Z}})^\tp+\overline{\pmb{Z}}\pmb{H}^\tp+\pmb{H}\pmb{H}^\tp\Vert_F^2 
+ 2\Vert \overline{\pmb{Z}}\Vert^2\cdot \Vert (\overline{\pmb{Z}})^\tp \pmb{D}(\pmb{H}+\overline{\pmb{Z}})\Vert_F^2\\
\overset{(iii)}{\leq}& 6\left( \Vert \pmb{H}\pmb{H}^\tp\Vert_F^2 + \Vert \pmb{H}(\overline{\pmb{Z}})^\tp\Vert_F^2 + \Vert \overline{\pmb{Z}}\pmb{H}^\tp\Vert_F^2 \right)\cdot\Vert \pmb{Z}\Vert^2
+2\Vert \overline{\pmb{Z}}\Vert^2\cdot \Vert (\overline{\pmb{Z}})^\tp \pmb{D}\pmb{H}\Vert_F^2\\
\leq& 6\left( \Vert \pmb{H}\Vert_F^4 + 2\Vert \pmb{H}\Vert_F^2 \Vert \overline{\pmb{Z}}\Vert^2 \right) \Vert \pmb{Z}\Vert^2 + 2\Vert \overline{\pmb{Z}}\Vert^2 \Vert (\overline{\pmb{Z}})^\tp \pmb{D}\pmb{H}\Vert_F^2.
\end{align*}
Here, the inequality $(iii)$ comes from basic inequality $(a+b+c)^2\leq3(a^2+b^2+c^2)$ and $(\overline{\pmb{Z}})^\tp \pmb{D}\overline{\pmb{Z}} =\pmb{0}$. Since $\Vert \overline{\pmb{Z}}\Vert = \sqrt{2\sigma_1^*}$, $\Vert \pmb{H}\Vert_F \leq \frac{1}{4}\sqrt{\sigma_r^*}$ and $\Vert \pmb{Z}\Vert \leq \Vert  \pmb{Z}-  \overline{\pmb{Z}}\Vert + \Vert \overline{\pmb{Z}}\Vert \leq \frac{1}{4}\sqrt{\sigma_r^*} + \sqrt{2\sigma_1^*}\leq\frac{7}{4}\sqrt{\sigma_1^*}$, it goes that
\begin{align}
\Vert \pmb{D}\pmb{Z}\pmb{Z}^\tp \pmb{D}\pmb{Z}\Vert_F^2  \leq& 6\left( \frac{1}{16}\sigma_r^* + 4\sigma_1^*\right)\cdot\left( \frac{7}{4}\sqrt{\sigma_1^*}\right)^2\cdot\Vert \pmb{H}\Vert_F^2 
+ 4\sigma_1^*\Vert (\overline{\pmb{Z}})^\tp \pmb{D}\pmb{H}\Vert_F^2. \label{lsc2}
\end{align}
Thus, combining (\ref{lsc1}), (\ref{lsc2}) and setting $\lambda = \frac{1}{2}$, we have
\begin{align}
\Vert \nabla  \mL_1(\pmb{Z})\Vert_F^2 
\leq& \left( 20\mu r\sigma_1^* + \frac{10}{16}\sigma_r^* + \frac{5}{16}\sigma_r^*\right) \cdot \left( 68\mu r\sigma_1^*\right) \cdot \Vert \pmb{H}\Vert_F^2 \nonumber\\
&+ 5\lambda^2 \Bigg[ 6\left( \frac{1}{16}\sigma_r^* + 4\sigma_1^*\right)\cdot\left( \frac{7}{4}\sqrt{\sigma_1^*}\right)^2\cdot\Vert \pmb{H}\Vert_F^2 
 + 4\sigma_1^*\Vert (\overline{\pmb{Z}})^\tp \pmb{D}\pmb{H}\Vert_F^2\Bigg]\nonumber\\
\leq& 1524\mu^2r^2\sigma_1^*{}^2\Vert \pmb{H}\Vert_F^2 + 5\sigma_1^*\Vert (\overline{\pmb{Z}})^\tp \pmb{D}\pmb{H}\Vert_F^2.\label{lsc3}
\end{align}
\end{proof}

\subsection*{A.3 Convergent Analysis}
\begin{proof}[Proof of Theorem \ref{t_1}]
Let $\pmb{H}^k = \pmb{Z}^k - \overline{\pmb{Z}}{}^k$ be the residual error. Based on the iteration
\begin{align*}
\pmb{Z}^{k+1} =  \mathcal{P}_{\mathcal{C}_1}\left(\pmb{Z}^k - \frac{\eta}{\Vert \pmb{Z}^*\Vert^2}\nabla  \mL_1(\pmb{Z}^k)\right),
\end{align*}
we have
\begin{align*}
\Vert \pmb{Z}^{k+1} - \overline{\pmb{Z}}{}^{k+1}\Vert_F^2 \leq& \Vert \pmb{Z}^{k+1} - \overline{\pmb{Z}}{}^{k}\Vert_F^2
 = \left\Vert  \mathcal{P}_{\mathcal{C}_1}\left(\pmb{Z}^k - \frac{\eta}{\Vert \pmb{Z}^*\Vert^2}\nabla  \mL_1(\pmb{Z}^k)\right) - \overline{\pmb{Z}}{}^{k} \right\Vert_F^2 \\
\leq& \left\Vert \pmb{Z}^k - \frac{\eta}{\Vert \pmb{Z}^*\Vert^2}\nabla  \mL_1(\pmb{Z}^k)- \overline{\pmb{Z}}{}^{k} \right\Vert_F^2
\end{align*}
for the utilization of Lemma \ref{l_3} and the projection $\mathcal{P}_{\mathcal{C}_1}$ is row-wise clipping. 
Based on Proposition \ref{p_1} and Proposition \ref{p_2} and by induction, we can know that the projection satisfies non-expansiveness property since $\dist(\pZ^0, \pZ^*) \leq \frac{1}{4}\sqrt{\sigma_r^*}$.
Then, it is easy to find that
\begin{align*}
\dist^2 (\pmb{Z}^{k+1}, \pmb{Z}^*) 
\leq& \left\Vert \pmb{Z}^k - \frac{\eta}{\Vert \pmb{Z}^*\Vert^2}\nabla  \mL_1(\pmb{Z}^k)- \overline{\pmb{Z}}{}^{k} \right\Vert_F^2\\
=& \Vert \pmb{H}^k\Vert_F^2 + \frac{\eta^2}{4\sigma_1^*{}^2}\Vert \nabla  \mL_1(\pmb{Z}^k) \Vert_F^2 - 2\langle \pmb{H}^k, \frac{\eta}{2\sigma_1^*}\nabla  \mL_1(\pmb{Z}^k) \rangle\\
\overset{(i)}{\leq}& \Vert \pmb{H}^k\Vert_F^2 + \frac{\eta^2}{4\sigma_1^*{}^2}\Vert \nabla  \mL_1(\pmb{Z}^k) \Vert_F^2 
- \frac{\eta}{\sigma_1^*}\left(\alpha \sigma^*_r\Vert\pmb{H}^k \Vert_F^2 + \beta \frac{\Vert \nabla  \mL_1(\pmb{Z}^k)\Vert_F^2}{\sigma_1^*}\right)\\
=& \left(1- \frac{\alpha \eta \sigma_r^*}{\sigma_1^*}\right) \Vert \pmb{H}^k\Vert_F^2 +  \left(\frac{\eta^2}{4\sigma_1^*{}^2} - \frac{\beta\eta}{\sigma_1^*{}^2}\right)\Vert \nabla  \mL_1(\pmb{Z}^k) \Vert_F^2\\
=&  \left(1-\frac{\alpha\eta}{\kappa}\right) \Vert \pmb{H}^k\Vert_F^2 + \frac{\eta}{\sigma_1^*{}^2} \left(\frac{1}{4}\eta - \beta\right)\Vert \nabla  \mL_1(\pmb{Z}^k) \Vert_F^2\\
\overset{(ii)}{\leq}& \left(1-\frac{\alpha\eta}{\kappa}\right) \Vert \pmb{H}^k\Vert_F^2,
\end{align*}
where the inequality (i) holds for (\ref{lsc5}) by setting $\alpha = \frac{1}{8}$, $\beta = \frac{1}{6096\mu^2 r^2 \kappa}$, the inequality (ii) holds for $\eta < \min \{\frac{1}{\alpha}, 4\beta \}$.\\
\indent Thus, the $k$-th step of Algorithm \ref{alg:pgd} satisfies
\begin{align*}
\dist(\pmb{Z}^k, \pmb{Z}^*) 
\leq& \sqrt{1-\frac{\eta}{8\kappa}} \dist(\pmb{Z}^{k-1}, \pmb{Z}^*) = (1-\frac{\eta}{8\kappa})^{\frac{k}{2}}\dist(\pmb{Z}^0, \pmb{Z}^*)  \\
\leq& \frac{1}{4}(1-\frac{\eta}{8\kappa})^{\frac{k}{2}} \sqrt{\sigma_r^*}
\end{align*}
with $\lambda = \frac{1}{2}$ in the model (\ref{lo2}) and $\eta \leq \frac{1}{1524\mu^2 r^2 \kappa}$.

Setting $\pH^k_{\pX} = \pX^k - \pX^*\pR$ and $\pH^k_{\pY} = \pY^k - \pY^*\pR$, 
\begin{align*}
\| \pX^k (\pY^k)^\tp - \pM^* \|_F 
\leq& \| \pH^k_{\pX}(\pH^k_{\pY})^\tp \|_F + \| \pH_{\pX}^k\pR^\tp(\pY^*)^\tp \|_F + \|  \pX^*\pR(\pH^k_{\pY})^\tp \|_F.
\end{align*}
Based on Theorem \ref{1}, 
\begin{align*}
\| \pH_{\pX}^k\pR^\tp(\pY^*)^\tp \|_F \leq \| \pY^*\|  \| \pR\|_F  \| \pH_{\pX}^k \|_F \leq \sqrt{\sigma_1^*} \| \pH_{\pX}^k \|_F
\end{align*}
because of $\| \pY^*\|= \| \pV^* (\pSig^*)^{\frac{1}{2}} \| \leq \sqrt{\sigma_1^*}\| \pV^*\| = \sqrt{\sigma_1^*}$. Similarly, we also get $\|  \pX^*\pR(\pH^k_{\pY})^\tp \|_F \leq \sqrt{\sigma_1^*} \| \pH_{\pY}^k \|_F$.

Meanwhile, 
\begin{align*}
\| \pH^k_{\pX}(\pH^k_{\pY})^\tp \|_F =& \frac{1}{2}\| \pH^k_{\pX}(\pH^k_{\pY})^\tp \|_F + \frac{1}{2}\| \pH^k_{\pX}(\pH^k_{\pY})^\tp \|_F \\
\leq& \frac{1}{8}\sqrt{\sigma_r^*} (\| \pH^k_{\pX}\|_F + \| \pH^k_{\pY}\|_F).
\end{align*}
Hence,
\begin{align*}
\| \pX^k (\pY^k)^\tp - \pM^* \|_F \leq& (\sqrt{\sigma_1^*} + \frac{1}{8}\sqrt{\sigma_r^*})(\| \pH^k_{\pX}\|_F + \| \pH^k_{\pY}\|_F)\\
\leq& \sqrt{2}(\sqrt{\sigma_1^*} + \frac{1}{8}\sqrt{\sigma_r^*}) \dist(\pZ^k, \pZ^*),
\end{align*}
 which completes the main theorem.
\end{proof}

\end{document}